\documentclass[10pt,letterpaper,reqno]{amsart}
\usepackage{amssymb,a4wide,amsthm,amsmath,amsfonts,xcolor}
\usepackage{mathrsfs}
\usepackage{color}
\usepackage{enumitem}
\usepackage[utf8]{inputenc}
\usepackage[T1]{fontenc}
\usepackage[colorlinks,citecolor=blue,urlcolor=blue]{hyperref}
\numberwithin{equation}{section}

\usepackage{bbm}

\newtheorem{theorem}{Theorem}[section]
\newtheorem{lemma}[theorem]{Lemma}
\newtheorem{proposition}[theorem]{Proposition}
\newtheorem{assumption}[theorem]{Assumption}
\newtheorem{corollary}[theorem]{Corollary}
\newtheorem{definition}[theorem]{Definition}
\newtheorem{main result}{Main Result}

\newtheorem{remark}[theorem]{Remark}

\newcommand\adda[1]{{\color{blue} #1}}

\newcommand\coma[1]{{\color{red} {#1}}}
\newcommand\dela[1]{}

\def\l{\left}
\def\r{\right}

\def\p{\prime}

\author[S. Gokhale]{Soham Gokhale}
\address{School of Mathematics\\ Indian Institute of Science Education and Research Thiruvananthapuram\\ Trivandrum 695551, INDIA}	\email{gokhalesoham16@iisertvm.ac.in}
\author[U.~Manna]{Utpal Manna}
\address{School of Mathematics\\ Indian Institute of Science Education and Research Thiruvananthapuram\\ Trivandrum 695551, INDIA}
\email{manna.utpal@iisertvm.ac.in}

\subjclass{60H}

\keywords{Landau Lifshitz Bloch equation, Young measures, optimal control, ferromagnetism}

\title{Optimal Control of The Stochastic Landau Lifshitz Bloch Equation}

\begin{document}
	\maketitle

	
	\begin{abstract}
		We consider the stochastic Landau-Lifshitz-Bloch equation in dimensions $1,2,3,$ perturbed by a real-valued Wiener process. We consider a Suslin space-valued control process with a general control operator, which can depend on both the control and the corresponding solution. We reduce the equation to a more general (relaxed) form, where the concept of Young measures is used. We then show the existence of a weak martingale solution to the controlled equation (relaxed). In the second part of the work, we show that for a general lower semicontinuous cost functional, the problem admits a weak relaxed optimal control. This is done using the theory of Young measures. Moreover, pathwise uniqueness is shown (for dimensions $1,2$), which implies the existence of a strong solution.
	\end{abstract}
	
	\section{Introduction}\label{section Introduction}

	Magnetic properties of materials have been a subject of special interest for hundreds of years. Their applications range from sailors using a simple magnetic needle in a compass to very complex imaging and recording devices. Data storage based on the magnetic properties of materials is one of the most used applications. The rise in the use of computers has increased the demand for having large storage, with efficient reading and writing. Generally speaking, a magnetic storage device consists of elements that can be magnetized in two separate directions (states) and hence can store one bit of data each. Therefore, reading and writing data on the said device includes switching the magnetization from one equilibrium state to another. The device's efficiency is, therefore, inherently dependent on the efficiency of the switching mechanism.

	Weiss (see \cite{brown1963micromagnetics} and references therein) initiated the study of the theory of ferromagnetism. It was further developed by Landau and Lifshitz (\cite{Landau+Lifshitz_1935_TTheoryOf_MagneticPermeability_Ferromagnetic}, \cite{landau1992theory}) and Gilbert \cite{Gilbert}. 
	
	For temperatures below the Curie temperature, magnetization can be modeled by the Landau-Lifshitz-Gilbert (LLG) equation. One of the drawbacks of the LLG equation is that it assumes the magnetization length to be a constant, which is unsuitable for higher temperatures, for example, in heat-assisted magnetic recording.
	Garanin (\cite{Garanin_1991_Generalized_EquationOfMotion_Ferromagnet}, \cite{garanin2004thermal}, \cite{garanin1997fokker}) developed a thermodynamically consistent approach, the Landau-Lifshitz-Bloch (LLB) equation for ferromagnetism. 
	The LLB equation is valid for temperatures both below and above the Curie temperature and essentially interpolates between the LLG equation at low temperatures and the Ginzburg-Landau theory for phase transitions.
	For instance, in light-induced demagnetization with powerful fs lasers, the electronic temperature is normally raised higher than the critical temperature. While LLG micromagnetics cannot work under these, micromagnetics based on the LLB equation has been shown to describe the observed fs magnetization dynamics \cite{garanin1997fokker}. For well posedness and more details about the LLB equation, we refer the reader to, for instance, Le \cite{LE_Deterministic_LLBE}, and references therein.
	
	Following the works \cite{ZB+BG+TJ_Weak_3d_SLLGE}, \cite{ZB+BG+Le_SLLBE}, we introduce the noise (and later, as an example, also control, see Remark \ref{remark particular examples}) in the effective field. One of the aims of studying the said models is to understand the transition phenomenon due to thermal fluctuations within the material \cite{brown_1979_ThermalFluctions}.
	Towards this, the noise is added to the energy of the system and hence comes down to the effective field, which is the negative gradient of the total energy. The thermal fluctuations can cause switching of the magnetization states and hence can change the data (in the case of magnetic data storage) without any external force. Hence in order to have more reliable data storage, it seems important to study and control the said switching phenomenon. Therefore, controlling the total energy (and thereby adding the control to the effective field) can help in a better understanding and designing of the model, and corresponding applications.\\
	The resulting stochastic Landau-Lifshitz-Bloch equation is the following. Let $T<\infty$ and let $h : \mathcal{O} \to \mathbb{R}^3$ be a given function. Let $m_0$ denote the initial data. Let $W$ denote a real-valued Wiener process on a given filtered probability space $\l( \Omega , \mathcal{F} , \mathbb{F} , \mathbb{P} \r)$. For a ferromagnetic domain $\mathcal{O}\subset\mathbb{R}^d,d=1,2,3$, the average spin polarization $m$ satisfies the following stochastic partial differential equation.
	\begin{align}
		\nonumber m(t) = & m_0 + \kappa_1 \int_{0}^{t} \l[ \Delta m(s) + \gamma m(s) \times \Delta m(s) - \kappa_2\l( 1 + \l| m(s) \r|_{\mathbb{R}^3}^2 \r) m(s) \r] \, ds  \\
		& + \int_{0}^{t} \gamma \l( m(s) \times h + \kappa_1 h \r)\, \circ dW(s),\ t\in[0,T].
	\end{align}

	Here the stochastic integral is understood in the Stratonovich sense, and hence the symbol "$ \circ dW(s)$". Neumann boundary conditions are assumed for the above problem.
	Regarding the accompanying constants, we have $\kappa_1$ as the damping parameter, $\kappa_2 = \frac{\kappa_1}{\chi_{||}}$, where $\chi_{||}$ denotes the longitudinal susceptibility, and $\gamma$ is the gyromagnetic ratio. For simplicity, we replace all the constants with $1$.
	
	Brze\'zniak and Le \cite{ZB+BG+Le_SLLBE} showed the existence of a martingale solution to the stochastic LLB equation. For dimensions $1,2$, they showed that the solution is pathwise unique, and hence the problem admits a unique, strong solution. They also show that for dimensions $1,2$, the problem admits\\
	invariant measures. On similar lines, Jiang \textit{et. al.} \cite{Jiang+Ju+Wang_MartingaleWeakSolnSLLBE} show the existence of a martingale solution for the stochastic LLB equation. Qiu \textit{et. al.} \cite{Qiu+Tang_Wang_AsymptoticBehaviourSLLBE} discuss the asymptotic behaviour of the said system. They establish large deviations principle and central limit theorem.
	Gokhale and Manna \cite{UM+SG_2022Preprint_SLLBE_LDP} (preprint) consider the stochastic LLB equation, perturbed by jump noise and establish a large deviations principle. The same authors in \cite{UM+SG_2022_SLLBE_WongZakai} show the existence of a solution for the stochastic LLB equation by Wong-Zakai type approximations. 
	As previously mentioned, one of the aims of studying the stochastic LLB equation is to study the transition phenomenon. It is desirable, for instance, in writing on magnetic recording devices, to have an external control that can switch the magnetization from one equilibrium state to another with optimal cost. This study might help in the better understanding and design of such a control.\\	
	Let $\mathbb{U}$ denote a metrizable Suslin space. This will be the control set that we consider. Let $X$ denote a separable Hilbert space which is densely and continuously embedded into the space $L^2$. Let $L: L^2 \times \mathbb{U} \to X$ denote the control operator, which depends (possibly non-linearly) on the control as well as the corresponding solution. The following problem is considered.\\
	\begin{align}\label{eqn problem considered original}
		\nonumber m(t) = & m_0 + \int_{0}^{t} \Delta m(s) + m(s) \times \Delta m(s) - \l( 1 + \l| m(s) \r|_{\mathbb{R}^3}^2 \r) m(s) \, ds \\
		& + \int_{0}^{t} L\big( m(s) , u(s) \big) \, ds + \int_{0}^{t} \l( m(s) \times h + h \r)\, \circ dW(s).
	\end{align}
	The assumptions on the operator $L$ are given in Assumption \ref{assumption Main Assumption}.\\
	\textbf{Cost functional:} For a control process $u$, with values in $\mathbb{U}$, and a corresponding solution $m$ on a given probability space with the Wiener process $W$, the cost functional $J$ is given by
	\begin{equation}\label{eqn cost functional}
		J\l( m , u \r) : = \mathbb{E}\l[ \int_{0}^{T} F\bigl(t,m\l(t\r) , u\l(t\r) \bigr) \, dt + \Psi\bigl(m\l(T\r)\bigr) \r].
	\end{equation}
	The running cost $F$ depends on both the solution and the control. $\Psi$ is the terminal cost. The assumptions on these are listed in Assumption \ref{assumption assumption on cost functional}. Note that henceforth we refer to $F$ as the cost instead of running cost. It should not be confused with $J$.

	Observe here that we do not assume any special conditions on the control operator $L$ with respect to the control process $u$, and no special convexity assumption is made on the cost functional $F$. With that in mind, we extend the space of admissible controls to a bigger space, called the space of admissible relaxed controls
	. This is done by the method of relaxation (Young \cite{Young_1942_GeneralizedSurfaces}, \cite{Young_1969Book_LecturesCalculusofVariations} and Warga \cite{Warga_1962_NecessaryConditionsRelaxedVariationalProblems}, \cite{Warga_1962_RelaxedVariationalProblems}). Ahmed \cite{Ahmed_1983_PropertiesRelaxedTrajectories}, Papageorgiou, \textit{et. al.} (\cite{Avgerinos+Papageorgiou_1990_OptimalControlAndRelaxation}, \cite{Papageorgiou_1989_ExistenceOptimalControl}, \cite{Papageorgiou_1989_RelaxationInfiniteDimensionControlSystem}) use relaxed controls in the evolution equations in Banach spaces, with Polish space valued controls (see also Lou \cite{Lou_2003_ExistenceOptimalControlSemilinearParabolicEqn}, \cite{Lou_2007_ExitenceNonExistenceOptimalControl}, Karoui, Nguyen, and Picque \cite{Karoui+Nguyen+Picque_1987_CompactificationMethodsExistenceOptimalControl}, Gatarek, Sobczyk \cite{Gatarek+Sobczyk_1994_OnExistenceOptimalControlsHilbertSpaceValuedDiffusions}, Haussmann, Lepeltier \cite{Haussmann+Lepeltier_2990_OnExistenceOptimalControl}).

	Regarding the stochastic case, Nagase and Nisio \cite{Nagase_ExistenceOptimalControlSPDE}, \cite{Nagase+Nisio_OptimalControlsSPDE} initiated the study of relaxed control for stochastic partial differential equations. Fleming and Nisio \cite{Fleming+Nisio_1984_OnStochasticRelaxedControl}, \cite{Fleming_1980_MeasureValuedProcessesControl} introduced Young measure technique for finite-dimensional stochastic systems. Sritharan \cite{Sritharan_DeterministicAndStochasticControl_SNSE} studied the optimal relaxed control of both deterministic and stochastic Navier-Stokes equations with linear and non-linear constitutive relations. He used the Minty-Browder technique (along with the martingale problem formulation of Stroock and Varadhan for the stochastic case) to establish the existence of optimal controls. Cutland and Grzesiak \cite{Cutland+Grzesiak_2005_OptimalControl3DSNSE}, \cite{Cutland+Grzesiak_2007_OptimalControl2DSNSE} study the existence of an optimal control for 2D (respectively 3D) Navier-Stokes equations. They use non-standard analysis (see also \cite{Berg+Neves_2007Book_StrengthNonstandardAnalysis}), along with relaxed control techniques. Brze\'zniak and Serrano \cite{ZB+RS} study an optimal relaxed control problem for a class of semilinear stochastic partial differential equations on Banach spaces. They use compactness properties of the class of Young measures on metrizable Suslin spaces (the control sets). Manna and Mukherjee \cite{UM+DM} use the technique of relaxed controls
	, with Suslin space-valued controls, combine them with the martingale problem formulation in the sense of Stroock and Varadhan to establish the existence of optimal controls. 

	Since the LLG equation and the LLB equation are similar in nature and model similar phenomena, it seems appropriate to mention some works regarding the stochastic LLG equation here. This is by no means an exhaustive list of works concerning the stochastic LLG equation. For more details, we refer the reader to \cite{ZB+BG+TJ_Weak_3d_SLLGE}, \cite{Laire_2021_RecentResultsLLEquation}, \cite{Lakshmanan_2011_FascinatingLLGE_Overview}, etc. and references therein.
	
	We write the stochastic LLG equation here for the reader's reference. The initial data is given by $m_0$. Neumann boundary conditions are assumed.
	\begin{align}
		\nonumber d m(t) = & \l[ \alpha \, m(t) \times \Delta m(t) - \gamma \, m(t) \times \l( m(t) \times m(t) \r) \r] \, dt \\
		& + \l[\alpha \, m(t) \times h - \gamma \, m(t) \times \l( m(t) \times h \r) \r] \circ dW(t), t\in[0,T].
	\end{align}
	Here $\alpha, \gamma$ are positive constants that depend on the gyromagnetic ratio and the damping parameter.
	Brze\'zniak, Goldys and Jegaraj in \cite{ZB+BG+TJ_Weak_3d_SLLGE} show the existence of a weak martingale solution to the stochastic LLG equation in dimensions $1,2,3$. Pu and Guo in \cite{Pu+Guo_2010} establish the existence of regular solutions to the stochastic LLG equation in dimension 2. Considering multidimensional noise and non-zero anisotropy energy, Brze\'zniak and Li in \cite{ZB+Li_Weak_Solution_SLLGE_Anisotropy} show the existence of weak solutions to the stochastic LLG equation. Goldys, Le and Tran in \cite{GB+Le+Tran_Finite_Element_Scheme_SLLGE} develop a finite element approximation scheme for the stochastic LLG equation, driven by a real-valued Wiener process. Goldys and Le, with Grotowski in \cite{BG+Grotowski+Le_Weak_martingale_SLLGE_multidimensional_noise} extend the work \cite{GB+Le+Tran_Finite_Element_Scheme_SLLGE} by considering multidimensional noise.
	In \cite{ZB+BG+TJ_LargeDeviations_LLGE} the authors establish a large deviations principle for the stochastic LLG equation in dimension 1.

	Following the lead of \cite{Dunst+Klein} (where a control problem for the LLG equation in dimension 1 is studied), Dunst \textit{et. al.} in \cite{D+M+P+V} show the existence of an optimal control for the stochastic LLG equation. Moreover, the authors also study a numerical solution of the optimal control problem. Using discretization techniques, convergence conditions for the approximate optimization problems are given. They have considered $H^1$ valued controls. The cost functional considered is similar to the one defined in \eqref{eqn cost functional}, with $F(m,u) = \l| m \r|_{L^2}^2 + \l| u \r|_{H^1}^2$.
	The authors (for the existence of an optimal control) do not consider the full noise, neglecting the triple product term. Brze\'zniak, Gokhale and Manna \cite{ZB+UM+SG_2022Preprint_SLLGE_Control} (preprint) plugged this gap by considering the full noise (for a Wiener process-driven stochastic LLG equation) in dimension 1. The existence of an optimal control, along with the existence of a maximal regular strong solution, is shown for the controlled equation. The controls are assumed to be $L^2$-valued and the cost functional is $F(m,u) = \l| m \r|_{H^1}^2 + \l| u \r|_{L^2}^2$, which is more natural as the solutions are $H^1$-valued. One can therefore see that the cost functional considered in this work is, in a sense, more general than the mentioned works.
	In \cite{Jensen+Majee+Prohl+Schellnegger_2019_DynamicProgramming_FiniteEnsembles}, the authors study magnetization dynamics in an ensemble of nanomagnetic particles, with thermal fluctuations. The study is done within the framework of the stochastic LLG equation, with anisotropic, stray, exchange and external magnetic fields.
	\dela{For the case where $\mathbb{U} = X = H^1$, 
		$L(m,u) =m \times u + m \times m \times u,$
		and \\
		$F(m,u) = \l|m\r|_{L^2}^2 + \l|u\r|_{H^1}^2$, we get the cost functional considered in the work of Prohl, \textit{et. al.} \cite{D+M+P+V} (for the stochastic LLG equation), with the control $u$ added to the effective field.}
	
	We now give a brief outline of the paper.
	Section \ref{section some notations, preliminaries} enlists some definitions and notations. In Section \ref{section the relaxed equation}, we give a relaxed formulation of equation \eqref{eqn problem considered original}, using Young measures. We also state some assumptions (Assumptions \ref{assumption Main Assumption}, \ref{assumption assumption on cost functional}) that are required on the control operator and the cost functional. Section \ref{section definitions and statements of the main results} enlists some definitions and statements of the main results. Under the assumptions given in Section \ref{section the relaxed equation}, we show (in Section \ref{section existence of a solution for the relaxed control}) that the relaxed equation \eqref{eqn problem considered relaxed} admits a weak martingale solution. The proof uses Faedo-Galerkin approximation, followed by some standard compactness results. For the control operator, we use some compactness results for Young measures. 
	In Appendix \ref{section pathwise uniqueness and existence of a strong solution}, we show that the obtained solution is pathwise unique (for $d = 1,2$), hence also showing the existence of a strong solution for the problem \eqref{eqn problem considered relaxed}. Section \ref{section existence of a optimal control} shows that under Assumptions \ref{assumption Main Assumption}, \ref{assumption assumption on cost functional}, the problem \ref{eqn problem considered original} admits a weak relaxed optimal control.
	

		\section{Some Notations, Preliminaries}\label{section some notations, preliminaries}
		We state some of the results that we have used in the proofs. We have mostly imitated the notations and borrowed results from \cite{ZB+RS}, \cite{Castaing_Book}, and \cite{Sritharan_DeterministicAndStochasticControl_SNSE}. For more details, we refer the reader to \cite{Balder_2000Book_LecturesYoungMeasureTheory}, \cite{Haussman+Lepeltier_1990_OnExistenceOptimalControls}, \cite{Jean+Jean_1981Book}, \cite{Fitte_2003_CompactnessCriteriaStableTopology},   \cite{Young_1969Book_LecturesCalculusofVariations}
		among others. 
		
		We assume that $\mathcal{O}\subset\mathbb{R}^d,\ d=1,2,3$ is a bounded domain with a smooth boundary. Only for Appendix \ref{section pathwise uniqueness and existence of a strong solution}, we restrict ourselves to $d = 1,2$. Let $0<T<\infty$ be arbitrary but fixed. Let $\mathbb{U}$ (the control set) denote a Hausdorff topological space. Let $\mathcal{B}\l(\mathbb{U}\r)$ denote the Borel $\sigma$-algebra on $\mathbb{U}$. Let $\mathcal{P}\l(\mathbb{U}\r)$ denote the set of all probability measures on $\mathcal{B}\l(\mathbb{U}\r)$. The set $\mathcal{P}\l(\mathbb{U}\r)$ is endowed with the $\sigma$-algebra generated by the projection maps
		\begin{equation}
			\pi_C : \mathcal{P}\l(\mathbb{U}\r) \ni q \mapsto q(C) \in [0,1] ,\ C\in \mathcal{B}\l(\mathbb{U}\r).
		\end{equation}
		Let $\l( \Omega , \mathcal{F} , \mathbb{F}, \mathbb{P} \r)$ denote a filtered probability space (satisfying the usual hypothesis) with filtration $\mathbb{F} = \l(\mathcal{F}_t\r)_{t\in[0,T]}$. Throughout the work, let $W$ denote a $\mathbb{R}$-valued Wiener process on the given probability space. Let $0 < T < \infty$ and the given function $h\in W^{2,\infty}$.
		\begin{definition}[Relaxed Control Process]\label{definition relaxed control process}
			A stochastic process $q = \l\{q_t\r\}_{t\in[0,T]}$ with values in $\mathcal{P}\l(\mathbb{U}\r)$ is called a  relaxed control process on $\mathbb{U}$ if the map
			\begin{equation}
				[0,T] \times \Omega \ni \l(t,\omega\r) \mapsto q_t\l( \omega , \cdot \r) \in \mathcal{P}\l( \mathbb{U} \r)
			\end{equation}
			is measurable. In other words, a relaxed control process on $\mathbb{U}$ is a measurable process with values in $\mathcal{P}\l( \mathbb{U} \r)$.
		\end{definition}
		\begin{definition}[Young Measure]\label{definition Young measure}
			Let $\lambda$ denote a non-negative $\sigma$-additive measure on $\mathcal{B}\bigl( \mathbb{U} \times [0,T]\bigr)$. We say that $\lambda$ is a Young measure on $\mathbb{U}$ if and only if $\lambda$ satisfies
			\begin{equation}
				\lambda\l( \mathbb{U} \times D \r) = \text{Leb}\l(D\r),\ \forall D\in\mathcal{B}\bigl([0,T]\bigr),
			\end{equation}
			with $\text{Leb}$ denoting the Lebesgue measure on $[0,T]$. The space of all Young measures on $\mathbb{U}$ will be denoted by $\mathcal{Y}\l(0,T: \mathbb{U} \r)$.
		\end{definition}

		\begin{lemma}[Disintegration of Young measures, Lemma 2.3, \cite{ZB+RS}]\label{lemma disintegration of Young measures}
			Let $\mathbb{U}$ be a Radon space. Let $\lambda:\Omega\to \mathcal{Y}\l(0,T : \mathbb{U} \r)$ be such that for every $J\in\mathcal{B}\bigl( \mathbb{U} \times [0,T]\bigr)$, the mapping
			\begin{equation}\label{eqn definition measurability of Young measure}
				\Omega \ni \omega \mapsto \lambda\l(\omega\r) \in [0,T]
			\end{equation}
			is measurable. Then there exists a  relaxed control process $q = \l\{q_t\r\}_{t\in[0,T]}$ on $\mathbb{U}$ such that
			\begin{equation}\label{eqn disintegration formula 1}
				\lambda\l( \omega , C \times D \r) = \int_{D} q_t\l(\omega , C\r) \, dt,\ \forall C\in \mathcal{B}\l(\mathbb{U}\r),\ D\in \mathcal{B}\bigl([0,T]\bigr).
			\end{equation}
		\end{lemma}
		The formula \eqref{eqn disintegration formula 1} is frequently denoted by
		\begin{equation}\label{eqn disintegration formula 2}
			\lambda\l( du,dt \r) = q_t\l( du \r)\, dt.
		\end{equation}
		
		Note that the above mapping \eqref{eqn definition measurability of Young measure} being measurable justifies calling $\lambda$ a random Young measure.
		
		\begin{definition}[Stable Topology]\label{definition stable topology}
			The stable topology on $\mathcal{Y}(0,T:\mathbb{U})$ is the weakest topology on $\mathcal{Y}(0,T:\mathbb{U})$ such that the mappings
			\begin{equation}
				\mathcal{Y}(0,T:\mathbb{U}) \ni \lambda \mapsto \int_{D} \int_{\mathbb{U}}f(u) \, \lambda(du,dt)\in \mathbb{R},
			\end{equation}
			is continuous for every $D\in\mathcal{B}([0,T])$ and $f\in C_{\text{b}}(\mathbb{U})$.
		\end{definition}
		
		\begin{proposition}
			The space of Young measures on a metrizable space (respectively metrizable Suslin) with the stable topology is metrizable (respectively metrizable Suslin).
		\end{proposition}
		
		\begin{definition}[Flexible Tighness]
			We say that a set $\mathscr{J}\subset \mathcal{Y}(0,T:\mathbb{U})$ is flexibly tight, if for each $\varepsilon>0$, there exists a measurable set-valued mapping $[0,T]\ni t \mapsto K_t\subset\mathbb{U}$ such that $K_t$ is compact for all $t\in[0,T]$ and
			\begin{equation}
				\sup_{\lambda\in\mathscr{J}}\int_{0}^{T}\int_{\mathbb{U}} \mathbbm{1}_{K_t^c}(u)\lambda(du,dt) < \varepsilon.
			\end{equation}
		\end{definition}
		
		\begin{definition}[Inf Compactness]
			A function $\eta:\mathbb{U} \to [0,\infty]$ is called inf-compact if and only if for every $R>0$ the level set
			\begin{equation*}
				\l\{ \eta \leq R \r\} = \l\{ u\in\mathbb{U} : \eta(u) \leq R \r\},
			\end{equation*}
			is compact.
		\end{definition}

		\begin{lemma}[Theorem 2.13, \cite{ZB+RS}
			]\label{lemma tightness}
			The following are equivalent for any $\mathscr{J}\subset \mathcal{Y}(0,T:\mathbb{U})$.
			\begin{enumerate}
				\item $\mathscr{J}$ is flexibly tight.
				
				\item There exists an inf compact function $\eta$ such that
				\begin{equation}
					\sup_{\lambda \in \mathscr{J}} \int_{0}^{T} \int_{\mathbb{U}} \eta(t,u) \lambda(du,dt) < \infty.
				\end{equation}
			\end{enumerate}
		\end{lemma}
		
		\begin{lemma}[Theorem 2.16, \cite{ZB+RS}]\label{lemma for convergence of integral of continuous bounded function Young measures}
			Let $\mathbb{U}$ be a metrizable Suslin space. If $\lambda_n \to \lambda$ stably in $\mathcal{Y}(0,T:\mathbb{U})$, then for every $f\in L^1(0,T:C_{\text{b}}(\mathbb{U}))$ we have
			\begin{equation}
				\lim_{n\to\infty} \int_{0}^{T}\int_{\mathbb{U}} f(t,u) \, \lambda_n(du,dt) =  \int_{0}^{T}\int_{\mathbb{U}} f(t,u) \, \lambda(du,dt).
			\end{equation}
		\end{lemma}

		\section{The Relaxed Equation}\label{section the relaxed equation}
		Following the discussion in Section \ref{section Introduction}, we replace the equation \eqref{eqn problem considered original} with the following relaxed equation. We primarily follow the works \cite{ZB+RS}, (see also \cite{D+M+P+V}, \cite{UM+DM}, and \cite{Sritharan_DeterministicAndStochasticControl_SNSE}).
		\begin{align}\label{eqn problem considered relaxed}
			\nonumber m(t) = & m_0 + \int_{0}^{t} \Delta m(s) + m(s) \times \Delta m(s) - \l( 1 + \l| m(s) \r|_{\mathbb{R}^3}^2 \r) \, ds \\
			& + \int_{0}^{t} \int_{\mathbb{U}} \l[ L(m(s),u) \r] \, \lambda(du,ds)  + \int_{0}^{t} \l( m(s) \times h + h \r)\, \circ dW(s).
		\end{align}
		Using the It\^o stochastic integral., the above equation can be understood as follows.
		\begin{align}
			\nonumber m(t) = & m_0 + \int_{0}^{t} \Delta m(s) + m(s) \times \Delta m(s) - \l( 1 + \l| m(s) \r|_{\mathbb{R}^3}^2 \r) \, ds \\
			& + \int_{0}^{t} \int_{\mathbb{U}} \l[ L(m(s),u) \r] \, \lambda(du,ds) + \frac{1}{2} \int_{0}^{t} \l( m(s) \times h \r) \times h \, ds
			+ \int_{0}^{t} \l( m(s) \times h + h \r)\,  dW(s).
		\end{align}
		Moreover, using Lemma \ref{lemma disintegration of Young measures} we can write the Young measure $\lambda$ using a relaxed control process $q$. Let $\phi$ be a test function (regularity specified later). The above equation can be understood in a weak (PDE) sense as follows.
		\begin{align}\label{eqn problem considered relaxed weak form with Ito integral}
			\nonumber \l\langle m(t) , \phi \r\rangle_{L^2}  = & \l\langle m_0 , \phi \r\rangle_{L^2} - \int_{0}^{t}  \l\langle \nabla m(s) , \nabla \phi \r\rangle_{L^2}  -  \l\langle m(s) \times \nabla m(s) , \nabla \phi \r\rangle_{L^2} \\
			\nonumber & -  \int_{\mathcal{O}} \l\langle \l( 1 + \l| m(s,x) \r|_{\mathbb{R}^3}^2 \r) m(s,x) , \phi \r\rangle_{\mathbb{R}^3} \, dx  \, ds  + \int_{0}^{t}  \l\langle \int_{\mathbb{U}}  L( m(s),u ) \, q_s(du) , \phi \r\rangle_{L^2}  \, ds \\
			& + \frac{1}{2} \int_{0}^{t}  \l\langle \bigl( m(s) \times h \bigr) \times h  , \phi \r\rangle_{L^2} \,  ds  + \int_{0}^{t}  \l\langle \bigl( m(s) \times h + h \bigr) , \phi \r\rangle_{L^2} \,  dW(s),\ t\in[0,T].
		\end{align}

		\begin{assumption}\label{assumption Main Assumption}{Assumptions on the control operator:}\hfill\\
			Recall that $X$ is a separable Hilbert space that is densely and continuously embedded into the space $L^2$.
			\begin{enumerate}
				
				\item 
				$L: L^2 \times \mathbb{U} \to X$.
				There exists an inf compact mapping $\kappa$ such that the following holds.
				\begin{equation}
					\l| L(m,v) \r|_{X} \leq C \l| m \r|_{L^2} \kappa (\cdot , v) + C \l| m \r|_{L^2} ,\ \forall v \in \mathbb{U}, \ m\in L^2.
				\end{equation}
				The second term on the right hand side of the above inequality does not change the calculations substantially. Therefore the term will not be considered henceforth.\\
				Moreover, for every Young measure $\lambda$, the function $\kappa$ satisfies the following bound.
				\begin{equation}
					\mathbb{E}\int_{0}^{T} \int_{\mathbb{U}} \kappa^2(s,v) \, \lambda(dv,ds)<\infty.
				\end{equation}
				For higher regularity, we assume the following bound on the function $\kappa$ above. For every random Young measure $\lambda$,
				\begin{equation}
					\mathbb{E}\int_{0}^{T} \int_{\mathbb{U}} \kappa^4(s,v) \,  \lambda(dv,ds)<\infty.
				\end{equation}

				\item 
				\begin{enumerate}
					\item The mapping $L$ is continuous in the first variable in the following sense. Let $R>0$. Let 
					$$ K_R = \l\{ v \in \mathbb{U} : \kappa(\cdot , v) \leq R\r\}.$$ Then for a sequence $w_n \to w$ in $L^2$, we have
					\begin{equation}
						\lim_{n\to\infty} \sup_{v\in K_R} \l| L(w_n , v) - L(w , v) \r|_{X} = 0.
					\end{equation}

					\item For each $w\in L^2$, the mapping
					\begin{equation}
						L(w,\cdot):\mathbb{U} \to L^2,
					\end{equation}
					is continuous.
					
					\dela{				\item The mapping $L$ is uniformly Lipschitz continuous (in the first variable).\adda{\\ (This is required for the existence part, wherein the finite dimensional FG Approximations require local Lipschitz coefficients).\\}
						\coma{Does assuming that $L$ is uniformly continuous in both the arguments suffice to say the above property?}
						\coma{\\OR\\
							Let $\lambda$ be a random Young measure on $\mathbb{U}$. with the disintegration $\lambda(dv,dt) = q_t(dv)dt$. Then the function
							\begin{equation}
								L^2\ni w \mapsto \int_{0} L(w,v) \, q_t(dv) \in L^2
							\end{equation}
							is Lipschitz continuous.}
					}
					
					\item Note that the following assumption is required only for ensuring that the finite dimensional approximation in Section \ref{section FG uniform energy estimates} admits a unique solution. For any random Young measure $\lambda$ with disintegration $\lambda(dv,dt) = q_t(dv)\,dt$, there exists an integrable function $f_{\lambda}: [0,T] \times \mathbb{U} \to \mathbb{R}$ such that for $w_1,w_2\in L^2$,
					\begin{equation}
						\l| L(w_1,v) - L(w_2,v) \r|_{X} \leq C \l| m_1 - m_2 \r|_{L^2}f_{\lambda}(\cdot,v).
					\end{equation}

				\end{enumerate}

				\item  Note that the following assumption is required specifically for the pathwise uniqueness of the weak martingale solution (see Theorem \ref{theorem pathwise uniqueness of solution}). There exists a constant $C>0$ such that for $m_1,m_2\in L^2$,
				\begin{equation}
					\l|L(m_1,v) - L(m_2,v)\r|_{X} \leq C \l|m_1 - m_2\r|_{L^2} \kappa(\cdot ,v),\ \forall v\in\mathbb{U},\ m\in L^2.
				\end{equation}

			\end{enumerate}	
		\end{assumption}

		\begin{assumption}{Assumptions on the cost functional:}\label{assumption assumption on cost functional}
			\begin{enumerate}
				\item $F:[0,T] \times H^1 \times \mathbb{U} \to [0,\infty]$ is assumed to be measurable in $t\in[0,T]$ and lower semicontinuous with respect to $\l( m , v \r) \in H^1 \times \mathbb{U}$.
				\item For the function $\kappa$ mentioned above, there exists a constant $C>0$ such that
				\begin{equation}\label{eqn coercivity assumption}
					F(t,m(t),v) \geq C \kappa^{4} (t,v),\ \forall v\in \mathbb{U},\ t\in[0,T],\forall m\in H^1.
				\end{equation}
				\item The terminal cost function $\Psi : H^1 \to \mathbb{R}$ is assumed to be lower semicontinuous.
			\end{enumerate}
		\end{assumption}

		\begin{remark}
			Assumption \ref{assumption Main Assumption} is given for a Hilbert space $X$ densely and continuously embedded into the space $L^2$. Since $X\hookrightarrow L^2\hookrightarrow X^\p$ is a Gelfand triple, we observe the following. For $v\in X,w\in L^2$,
			\begin{equation}
				\ _{X}\l\langle v , w \r\rangle_{X^\p}  = \l\langle v , w \r\rangle_{L^2}.
			\end{equation}
			Moreover, there exists constants $C_1,C_2 > 0$ such that for $v\in X$,
			\begin{equation}
				\l| v \r|_{X^\p} \leq C_1 \l| v \r|_{L^2} \leq C_2 \l| v \r|_{X}.
			\end{equation}
			With the above observations in mind, we show the calculations only for the special case $X = L^2$.
		\end{remark}

		\begin{remark}[Particular Examples]\label{remark particular examples} \hfill
			\begin{enumerate}			
				\item
				For $\mathbb{U} = X = L^2$, with $$L(m,u) = m \times u + u,$$ we get the control operator when the control $u$ is simply added to the effective field .
				Considering the cost functional $F(m,u) = \l|m\r|_{H^1}^2 + \l|u\r|_{L^2}^2$ will give an optimal control (relaxed) for a simplified stochastic LLB control equation. Regarding the inf compact funtion, it suffices to have $\kappa(t,u) = \l| u(t) \r|_{H^2}$. (By Example 2.12, \cite{ZB+RS}, we have that $\kappa$ is inf compact).			
			\end{enumerate}
		\end{remark}

		\section{Definitions and Statements of the main results}\label{section definitions and statements of the main results}

		\begin{definition}[Weak martingale solution to the relaxed equation \eqref{eqn problem considered relaxed}]\label{definition weak martingale solution relaxed equation}
			Let $\lambda$ be a given random Young measure on the space $\mathbb{U}$. A tuple 
			$ \tilde{\pi} = \l( \tilde{\Omega} , \tilde{\mathcal{F}} , \tilde{\mathbb{F}} , \tilde{\mathbb{P}} , \tilde{m} , \tilde{\lambda} , \tilde{W} \r)$ is said to be a weak martingale solution of the problem \eqref{eqn problem considered relaxed} if the following hold.
			\begin{enumerate}
				\item The tuple $\l( \tilde{\Omega} , \tilde{\mathcal{F}} , \tilde{\mathbb{F}} , \tilde{\mathbb{P}} \r)$ is a filtered probability space that satisfies the usual conditions, with the filtration $\tilde{\mathbb{F}} = \l\{ \tilde{\mathcal{F}}_t \r\}_{t\in[0,T]}$ and $\sigma$-algebra $\tilde{\mathcal{F}}$.
				\item The process $\tilde{W}$ is a real valued Wiener process on the above mentioned probability space.
				\item $\tilde{\lambda}$ is a random Young measure on the space $\mathbb{U}$ with the same law as that of $\lambda$ on the space $\mathcal{Y}(0,T:\mathbb{U})$. $\mathbb{U}$ 
				\item The process $\tilde{m}$ is an $H^1$-valued progressively measurable such that
				\begin{equation}
					\tilde{\mathbb{E}} \sup_{t\in[0,T]}\l| \tilde{m}(t) \r|_{H^1}^2 < \infty
				\end{equation}
				\begin{equation}
					\tilde{\mathbb{E}} \int_{0}^{T} \l| \tilde{m}(t) \r|_{H^2}^2 \, dt < \infty
				\end{equation}
				\item The equality \eqref{eqn problem considered relaxed weak form with Ito integral} holds (with $\tilde{m}, \tilde{\lambda}$ and $\tilde{W}$) for each $\phi\in L^4(\tilde{\Omega}:H^1)$ and $t\in[0,T]$.
			\end{enumerate}
		\end{definition}	
		Corresponding to the cost functional $J$ defined in \eqref{eqn cost functional}, we define the relaxed cost functional $\mathcal{J}$ as follows, for an admissible solution $\pi\in\mathcal{U}(m_0,T)$, $ \pi = \l( \Omega , \mathcal{F} , \mathbb{F} , \mathbb{P} , m , \lambda , W \r)$.
		\begin{equation}\label{eqn relaxed cost functional}
			\mathcal{J}\l( \pi  \r) : = \mathbb{E}\l[ \int_{0}^{T} \int_{\mathbb{U}} F\bigl(t,m(t) , u \bigr) \, \lambda(du,dt) + \Psi\bigl(m(T)\bigr) \r]
		\end{equation}
		The space of admissible solutions $\mathcal{U}(m_0,T)$ is the collection of all weak martingale solutions to the problem \eqref{eqn problem considered relaxed}, with finite cost $\mathcal{J}$. \\
		Let us denote by $\Lambda$, the infimum of $\mathcal{J}$ over the space of admissible solutions $\mathcal{U}(m_0,T)$.

		\begin{theorem}[Existence of weak martingale solution to \eqref{eqn problem considered relaxed}]\label{theorem existence of weak martingale solution}
			Let $\lambda$ be a random Young measure on the space $\mathbb{U}$. Then the equation \eqref{eqn problem considered relaxed} admits a weak martingale solution as defined in \ref{definition weak martingale solution relaxed equation}. Moreover, $\lambda$ and $\tilde{\lambda}$ have the same laws on the space $\mathcal{Y}(0,T:\mathbb{U})$.
		\end{theorem}
		We give a proof for this result in Section \ref{section existence of a solution for the relaxed control}.
		
		\begin{definition}[Weak relaxed optimal control]\label{definition weak relaxed optimal control}
			A weak martingale solution $\tilde{\pi}$ is said to be a weak relaxed optimal control to the problem \eqref{eqn problem considered original} with the cost functional $\mathcal{J}$ if
			\begin{equation}
				\mathcal{J}(\tilde{\pi}) = \Lambda.
			\end{equation}
		\end{definition}
		
		\begin{theorem}[Existence of a weak relaxed optimal control]\label{theorem existence of weak relaxed optimal control}
			Let the Assumptions \ref{assumption Main Assumption} and \ref{assumption assumption on cost functional} hold. Then the problem \eqref{eqn problem considered original} admits a weak relaxed optimal control as in Definition \ref{definition weak relaxed optimal control}.
		\end{theorem}
		Proof for the above result is given in Section \ref{section existence of a optimal control}.
		
		\dela{\coma{\begin{remark}[Relation of the \textbf{relaxed} and the \textbf{original} Equation:]\label{remark relation between relaxed and original equation}
					In particular, for a given control process $u(t),t\in[0,T]$, with values in $\mathbb{U}$, for $q_t = \delta_{u(t)}$, we obtain the original equation \eqref{eqn problem considered original}. Note that this may require some conditions \textbf{(measurability, etc.)} on the operator $L$ and possibly the process $u$. In particular, for $q_t = \delta_{u(t)}$, we obtain the original equation. Here, for each $t\in[0,T]$, $\delta_{u(t)}$ denotes the Dirac Delta measure and is $1$ at $u(t)$ and $0$ otherwise.
				\end{remark}
			}
		}

		\section{Existence of a solution for the Relaxed Control Problem}\label{section existence of a solution for the relaxed control}

		\begin{proof}[Idea of the proof of Theorem \ref{theorem existence of weak martingale solution}]
			We start by approximating the equation \eqref{eqn problem considered relaxed} in finite dimensions, using the Faedo-Galerkin approximations (see equation \eqref{eqn problem considered relaxed weak form with Ito integral}). We follow that up with establishing uniform estimates on the obtained approximations. Then using some standard compactness results, we obtain convergence of the sequence of approximates (possibly along a subsequence) to a limit, which is a solution to the problem \eqref{eqn problem considered relaxed} in the sense of Definition \ref{definition weak martingale solution relaxed equation}.
			
		\end{proof}
		
			%

		Let $A = -\Delta$ be the Neumann Laplacian operator.
		Let $\{e_i\}_{i\in\mathbb{N}}$ be an orthonormal basis of $L^2$ consisting of eigen functions of $A$ (for example refer to \cite{Evans} page 335 Theorem 1). Let $H_n$ denote the span of the first $n$ eigen functions.	
		The equation \eqref{eqn problem considered relaxed weak form with Ito integral} is approximated ($\mathbb{P}$-a.s.) with the following equation in $H_n$.
		\begin{align}\label{eqn problem considered relaxed FG approximation}
			\nonumber m_n(t) = & P_nm_0 + \int_{0}^{t} P_n\bigl(\Delta m_n(s)\bigr) + P_n\l( m_n(s) \times \Delta m_n(s) \r) - P_n\l( \l( 1 + \l| m_n(s) \r|_{\mathbb{R}^3}^2 \r) m_n(s) \r) \, ds \\
			\nonumber & + \int_{0}^{t} P_n\l(\int_{\mathbb{U}} L(m_n(s),v) \, q_s(dv) \r) \, ds\\
			& + \frac{1}{2} \int_{0}^{t} P_n\bigl( \l( m_n(s) \times h \r) \times h \bigr) + \int_{0}^{t} P_n\bigl( m_n(s) \times h + h \bigr)\,  dW(s).
		\end{align}
		
		We define the following operators.
		\begin{align}
			F_n^1 & : H_n \ni m \mapsto P_n\l(\Delta m\r),\\
			F_n^2 & : H_n \ni m \mapsto P_n\l( m \times \Delta m\r),\\
			F_n^3 & : H_n \ni m \mapsto P_n\l( \l( 1 + \l| m \r|_{\mathbb{R}^3}^2\r) m\r),\\
			F_n^4 & : H_n \ni m \mapsto P_n\bigl( \l( m \times h \r) \times h \bigr),\\
			F_n^5 & : H_n \ni m \mapsto P_n \l(\int_{\mathbb{U}} L(m,v) \, q_{\cdot}(dv)  \r) , \\
			g_n & : H_n \ni m \mapsto P_n\l( m \times h + h \r).
		\end{align}
		
				%
				%
					%
					%
		All coefficients in \eqref{eqn problem considered relaxed FG approximation} are locally Lipschitz (see, for example \cite{ZB+BG+Le_SLLBE}). For the control term $(F_n^5)$, we use Assumption \ref{assumption Main Assumption}.
		Linear growth follows from Lipschitz continuity. Hence for each $n\in\mathbb{N}$, the problem \eqref{eqn problem considered relaxed FG approximation} admits a unique solution $m_n\in C([0,T]:H_n)$.
		
		\subsection{Uniform Energy Estimates}\label{section FG uniform energy estimates}
		We now establish some uniform energy estimates for the approximations.
		\begin{lemma}\label{lemma FG bounds 1}
			There exists a constant $C>0$, which is independent of $n$ such that the following hold.
			\begin{equation}\label{eqn FG L infinity L2 bound mn}
				\mathbb{E}\sup_{t\in[0,T]} \l| m_n(t) \r|_{L^2}^2 \leq C,
			\end{equation}
			\begin{equation}\label{eqn FG L2 H1 bound mn}
				\mathbb{E}\int_{0}^{T} \l| m_n(t) \r|_{H^1}^2 \, dt \leq C,
			\end{equation}
			\begin{equation}\label{eqn FG L4 L4 bound mn}
				\mathbb{E}\int_{0}^{T} \l| m_n(t) \r|_{L^4}^4 \, dt \leq C.
			\end{equation}
			Note that the constant $C$ depends on the bound $\mathbb{E}\int_{0}^{T} \int_{\mathbb{U}} \kappa(s,v)^4 \lambda(dv,ds)$, the given function $h$, the initial data $m_0$ and the time $T$ and the domain $\mathcal{O}$.
		\end{lemma}
		\begin{proof}[Proof of Lemma \ref{lemma FG bounds 1}]
			We give a sketch of the proof. For detailed arguments for a similar proof, we refer the reader to \cite{ZB+BG+Le_SLLBE}.
			Consider the equality \eqref{eqn problem considered relaxed FG approximation}.		
			Applying the It\^o formula for $\l\{H_n \ni v \mapsto \frac{1}{2} \l|  v \r|_{L^2}^2 \in \mathbb{R}\r\}$ gives, for $t\in[0,T]$
			\begin{align}\label{eqn FG bounds lemma 1 inequality 1}
				\nonumber \frac{1}{2} \l| m_n(t) \r|_{L^2}^2 = & \frac{1}{2} \l| P_nm_0 \r|_{L^2}^2  +  \int_{0}^{t} \bigg[ \big\langle P_n\bigl(\Delta m_n(s)\bigr) + P_n\l( m_n(s) \times \Delta m_n(s) \r) \\
				\nonumber & - P_n\l( \l( 1 + \l| m_n(s) \r|_{\mathbb{R}^3}^2 \r) m_n(s) \r) , m_n(s) \big\rangle_{L^2} \bigg] \, ds \\
				\nonumber & + \int_{0}^{t} \l\langle P_n\l(\int_{\mathbb{U}}  L(m_n(s),v)  \, q_s(dv) \r) , m_n(s)\r\rangle_{L^2} \, ds \\
				\nonumber & + \frac{1}{2} \int_{0}^{t} \l\langle P_n\bigl( \l( m_n(s) \times h \r) \times h \bigr) , m_n(s)\r\rangle_{L^2} \, ds \\
				\nonumber & + \int_{0}^{t} \l\langle P_n\bigl( m_n(s) \times h + h \bigr) , m_n(s)\r\rangle_{L^2}  \,  dW(s)\\
				= & \frac{1}{2} \l| P_nm_0 \r|_{L^2}^2 + \sum_{i=1}^{4} I_i(t) .
			\end{align}
			\textbf{Calculation for $I_1$:}\\
			Using integration by parts, \cite{ZB+BG+Le_SLLBE} we can obtain the following inequality (See \cite{LE_Deterministic_LLBE}).
			\begin{align}\label{eqn bounds lemma 1 inequality for I1}
				I_1(t) =  \dela{& \coma{- \int_{0}^{t} \l| \nabla m_n(s) \r|_{L^2}^2 \, ds -  \int_{0}^{t} \l| m_n(s) \r|_{L^2}^2 \, ds -  \int_{0}^{t} \l|  m_n(s) \r|_{L^4}^4 \, ds }\\		
					= &} - \int_{0}^{t} \l| m_n(s) \r|_{H^1}^2 \, ds -  \int_{0}^{t} \l|  m_n(s) \r|_{L^4}^4 \, ds \leq 0.
			\end{align}
			\textbf{Calculation for $I_2$:}\\
			We now show calculations for the control term.
			Let $\varepsilon>0$. By using the continuous embedding $L^4\hookrightarrow L^2$, Young's inequality, and integrating over time, there exists a constant $C>0$ such that
			\begin{align}
				\int_{0}^{t} \l| \int_{\mathbb{U}} \l\langle L(m_n , v) , m_n(s) \r\rangle_{L^2} \, q_{s}(dv) \, ds \r| \leq \frac{\varepsilon}{2} \int_{0}^{t} \l| m_n(s) \r|_{L^4}^4 \, ds + \frac{C^4}{4\varepsilon^2} \int_{0}^{t} \int_{\mathbb{U}} \kappa^2(s , v) \, q_{s}(dv) \, ds.
			\end{align}		
			We choose $\varepsilon$ small enough ($\varepsilon = 1$ works here) such that the first term on the right hand side can be adjusted with the second term on the right hand side of \eqref{eqn bounds lemma 1 inequality for I1}, still keeping the coefficient negative.
			
			
			For the last term on the right hand side of \eqref{eqn FG bounds lemma 1 inequality 1}, we have the following by the Burkh\"older-Davis-Gundy inequality. Let $\delta > 0$. Then, there exists a constant (which can depend on $\delta,h$ but not on $n\in\mathbb{N}$) such that the following holds.
			\begin{align}\label{eqn FG bounds lemma 1 inequality for noise term}
				\mathbb{E} \sup_{t\in[0,T]}  \l| \int_{0}^{t} \l\langle P_n\bigl( m_n(s) \times h + h \bigr) , m_n(s)\r\rangle_{L^2}  \,  dW(s) \r| 
				\leq  \frac{\delta}{2} \mathbb{E} \sup_{t\in[0,T]} \l| m_n(t) \r|_{L^2}^2 + \frac{C}{2\delta}.
			\end{align}
			Combining the above calculations (except for \eqref{eqn FG bounds lemma 1 inequality for noise term}) with \eqref{eqn FG bounds lemma 1 inequality 1} gives (for a suitable constant $C>0$) the following.
			\begin{align}\label{eqn FG bounds lemma 1 inequality 3}
				\nonumber \frac{1}{2} \l| m_n(t) \r|_{L^2}^2 + & \int_{0}^{t} \l| m_n(s) \r|_{H^1}^2 \, ds + \int_{0}^{t} \l|  m_n(s) \r|_{L^4}^4 \, ds \\
				\nonumber & \leq  \frac{1}{2} \l| P_n m_0 \r|_{L^2}^2 + C  \int_{0}^{t}  \int_{\mathbb{U}} \kappa^2(s,u) \, q_s(du)    \, ds + C \int_{0}^{t} \l| m_n(s) \r|_{L^2}^2 \, ds \\
				& + \l| \int_{0}^{t} \l\langle P_n\bigl( m_n(s) \times h + h \bigr) , m_n(s)\r\rangle_{L^2}  \,  dW(s) \r|.
			\end{align}	
			All the terms on the left hand side of the above inequality are non-negative. Hence the second and the third term can be neglected for now, still preserving the inequality. Taking the supremum over $[0,T]$ of the resulting inequality, followed by taking the expectation of both sides, followed by the inequality \eqref{eqn FG bounds lemma 1 inequality for noise term} gives the following.
			
			\begin{align}\label{eqn FG bounds lemma 1 inequality 4}
				\nonumber \mathbb{E} \sup_{t\in[0,T]} \l| m_n(t) \r|_{L^2}^2  \leq & \mathbb{E} \l| P_n m_0 \r|_{L^2}^2 + 2 C  \mathbb{E}  \int_{0}^{T}  \int_{\mathbb{U}} \kappa^2(s,u) \, q_s(du)    \, ds 
				+ 2 C  \mathbb{E} \int_{0}^{T} \l| m_n(s) \r|_{L^2}^2 \, ds \\
				& +  \delta \mathbb{E} \sup_{t\in[0,T]} \l| m_n(t) \r|_{L^2}^2 + \frac{C}{\delta}.
			\end{align}
			We choose $\delta \leq 1$, so that the last term on the right hand side can be balanced with the left hand side of \eqref{eqn FG bounds lemma 1 inequality 4}. In particular, we choose $\delta = \frac{1}{2}$.
			\dela{ to give
				\begin{align}\label{eqn FG bounds lemma 1 inequality 5}
					\frac{1}{2} \mathbb{E} \sup_{t\in[0,T]} \l| m_n(t) \r|_{L^2}^2 & \leq  \mathbb{E} \l| P_n m_0 \r|_{L^2}^2 + 2 C  \mathbb{E}  \int_{0}^{T}  \int_{\mathbb{U}} \kappa^2(s,u) \, q_s(du)    \, dt 
					+ 2 C  \mathbb{E} \int_{0}^{T} \l| m_n(s) \r|_{L^2}^2 \, ds + C.
				\end{align}
			}
			By Assumption \ref{assumption Main Assumption}, there exists a constant $C>0$ such that
			\begin{align}\label{eqn FG bounds lemma 1 inequality 6}
				\mathbb{E} \sup_{t\in[0,T]} \l| m_n(t) \r|_{L^2}^2 & \leq  C \l( 1 + \mathbb{E} \l| P_n m_0 \r|_{L^2}^2 \r)
				+ C  \mathbb{E} \int_{0}^{T} \l| m_n(s) \r|_{L^2}^2 \, ds + C.
			\end{align}
			Using the Gronwall inequality gives the required result \eqref{eqn FG L infinity L2 bound mn}.\\
			We now go back to the inequality \eqref{eqn FG bounds lemma 1 inequality 3}. The first and third terms are non-negative. Hence they can be neglected, still preserving the inequality. Taking the supremum over $[0,T]$, followed by taking the expectation of both sides, and then using \eqref{eqn FG L infinity L2 bound mn} gives \eqref{eqn FG L2 H1 bound mn}. The bound \eqref{eqn FG L4 L4 bound mn} can be  obtained similarly by observing that the first two terms on the left hand side of \eqref{eqn FG bounds lemma 1 inequality 3} are non-negative.
		\end{proof}

		\begin{lemma}\label{lemma FG bounds 2}
			There exists a constant $C>0$, which is independent of $n$ such that the following hold.
			\begin{equation}\label{eqn FG L infinity H1 bound mn}
				\mathbb{E}\sup_{t\in[0,T]} \l| m_n(t) \r|_{H^1}^2 \leq C,
			\end{equation}
			\begin{equation}\label{eqn FG L2 H2 bound mn}
				\mathbb{E}\int_{0}^{T} \l| \Delta m_n(t) \r|_{L^2}^2 \, dt \leq C.
			\end{equation}
			Note that the constant $C$ depends on the bound $\mathbb{E}\int_{0}^{T} \int_{\mathbb{U}} \kappa(s,v)^4 \lambda(dv,ds)$, the given $h$, the initial data $m_0$ and the time $T$.
		\end{lemma}
		\begin{proof}[Proof of Lemma \ref{lemma FG bounds 2}]
			
			Consider the equality \eqref{eqn problem considered relaxed FG approximation}.
			
			Applying the It\^o formula for $\l\{H_n \ni v \mapsto \frac{1}{2} \l| \nabla v \r|_{L^2}^2 \in \mathbb{R}\r\}$ 
			gives
			\begin{align}\label{eqn FG bounds lemma 2 inequality 1}
				\nonumber \frac{1}{2} \l| \nabla m_n(t) \r|_{L^2}^2 = & \frac{1}{2} \l| \nabla P_n m_0 \r|_{L^2}^2  +  \int_{0}^{t} \bigg[ \big\langle P_n\bigl(\Delta m_n(s)\bigr) + P_n\l( m_n(s) \times \Delta m_n(s) \r) \\
				\nonumber & - P_n\l( \l( 1 + \l| m_n(s) \r|_{\mathbb{R}^3}^2 \r) m_n(s) \r) , - \Delta m_n(s) \big\rangle_{L^2} \bigg] \, ds \\
				\nonumber & + \int_{0}^{t} \l\langle P_n\l(\int_{\mathbb{U}}  L(m_n(s),v)  \, q_s(dv) \r) , \Delta m_n(s)\r\rangle_{L^2} \, ds \\
				\nonumber & + \frac{1}{2} \int_{0}^{t} \l\langle P_n\bigl( \l( m_n(s) \times h \r) \times h \bigr) , - \Delta m_n(s)\r\rangle_{L^2} \, ds \\
				\nonumber & + \int_{0}^{t} \l\langle P_n\bigl( m_n(s) \times h + h \bigr) , - \Delta m_n(s) \r\rangle_{L^2}  \,  dW(s)\\
				= & \frac{1}{2} \l| P_nm_0 \r|_{L^2}^2 + \sum_{i=1}^{4} I_i(t) .
			\end{align}
			\textbf{Calculation for $I_1$:}\\
			For $I_1$, we have the following inequality (see \cite{LE_Deterministic_LLBE}) for $t\in [0,T]$.
			\begin{align}\label{eqn bounds lemma 2 calculation for I1}
				\nonumber I_1(t) = & - \int_{0}^{t} \l| \Delta m_n(s) \r|_{L^2}^2 \, ds - \int_{0}^{t} \l| \nabla m_n(s) \r|_{L^2}^2 \, ds \\
				& - 2 \int_{0}^{t} \l\langle m_n(s) , \nabla m_n(s) \r\rangle_{L^2}^2 \, ds
				- \int_{0}^{t} \l| m_n(s) \r|_{\mathbb{R}^3}^2 \l| \nabla m_n(s) \r|_{L^2}^2 \, ds \leq 0.
			\end{align}
			\textbf{Calculations for $I_2$:}\\
			Let $\varepsilon>0$. By Young's inequality, the continuous embedding $L^4\hookrightarrow L^2$ and the Assumption \ref{assumption Main Assumption}, we can show the following inequality for some constant $C>0$.		
			\begin{equation}
				I_2(t) \leq  \frac{C}{\varepsilon^2} \int_{0}^{t} \l| m_n(s) \r|_{L^4}^4 \, ds +  \int_{0}^{t}  \int_{\mathbb{U}} \kappa^4(s,v) \, q_s(du)  \, ds + \frac{\varepsilon}{2} \int_{0}^{t} \l|\Delta m_n(s) \r|_{L^2}^2 \, ds .
			\end{equation}
			We later choose this $\varepsilon$ small enough so that the term with $\varepsilon$ can be balanced with the $(\Delta m_n)$ term from \eqref{eqn bounds lemma 2 calculation for I1}, still keeping the coefficient negative. For example, $\varepsilon = 1$ works here.\\
			\textbf{Calculation for $I_3$:}\\
			There exists a constant $C(h)>0$ such that
			\begin{align}
				I_3(t) \leq C(h) \int_{0}^{t} \l| m_n(s) \r|_{L^2}^2 \, ds + \frac{1}{2} \int_{0}^{t} \l| \nabla m_n(s) \r|_{L^2}^2 \, ds.
			\end{align}		
			Combining the above calculations with \eqref{eqn FG bounds lemma 2 inequality 1} gives, for $\varepsilon = \frac{1}{2}$ (and an appropriate constant $C>0$),
			\begin{align}\label{eqn FG bounds lemma 2 inequality 3}
				\nonumber \l| \nabla m_n(t) \r|_{L^2}^2 + \int_{0}^{t} \l| \Delta m_n(s) \r|_{L^2}^2 \, ds 
				\leq & C \l(  1 +  \l| m_0 \r|_{H^1}^2  \r)
				+ C \int_{0}^{t} \l| m_n(s) \r|_{L^4}^4 \, ds +  \int_{0}^{t}  \int_{\mathbb{U}} \kappa^4(s,v) \, q_s(du)  \, ds \\
				\nonumber & + C \int_{0}^{t} \l| m_n(s) \r|_{L^2}^2 \, ds 
				+ C \int_{0}^{t} \l| \nabla m_n(s) \r|_{L^2}^2 \, ds \\
				& + \l| \int_{0}^{t} \l\langle P_n\bigl( m_n(s) \times h + h \bigr) , - \Delta m_n(s) \r\rangle_{L^2}  \,  dW(s) \r|.
			\end{align}
			\dela{The second term on the left hand side of the above inequality is non-negative. Hence it can be neglected for now, still preserving the inequality. Taking the supremum over $[0,T]$ followed by taking the expectation of both sides gives
				\begin{align}\label{eqn FG bounds lemma 2 inequality 4}
					\nonumber \mathbb{E} \sup_{t\in[0,T]} \l| \nabla m_n(t) \r|_{L^2}^2 \leq & C \l(  1 + \mathbb{E} \l| m_0 \r|_{H^1}^2  \r)
					+ C \mathbb{E} \int_{0}^{T} \l| m_n(s) \r|_{L^4}^4 \, ds \\
					\nonumber & + C \int_{0}^{t}  \int_{\mathbb{U}} \kappa^4(s,v) \, q_s(du)  \, ds + C \mathbb{E} \int_{0}^{T} \l| m_n \r|_{L^2}^2 \, ds 
					+ C \mathbb{E} \int_{0}^{T} \l| \nabla m_n \r|_{L^2}^2 \, ds \\
					& + \mathbb{E} \sup_{t\in[0,T]} \l| \int_{0}^{t} \l\langle P_n\bigl( m_n(s) \times h + h \bigr) , - \Delta m_n(s) \r\rangle_{L^2}  \,  dW(s) \r|.
				\end{align}
			}
			For the last term, we have the following by the Burkh\"older-Davis-Gundy inequality.\\
			Let $\delta>0$. There exists a constant $C>0$ independent of $n\in\mathbb{N}$ such that
			\begin{align}
				\nonumber \mathbb{E} \sup_{t\in[0,T]} \l| \int_{0}^{t} \l\langle P_n\bigl( m_n(s) \times h + h \bigr) , - \Delta m_n(s) \r\rangle_{L^2}  \,  dW(s) \r|
				\nonumber  \leq & \frac{\delta}{2} \mathbb{E}   \sup_{t\in[0,T]} \l| \nabla m_n(t) \r|_{L^2}^2 \\
				& + \frac{2}{\delta} C \mathbb{E} \int_{0}^{T} \l| m_n(s) \r|_{L^2}^2  \, ds + C.
			\end{align}
			Note that the second term on the left hand side of \eqref{eqn FG bounds lemma 2 inequality 3} is non-negative, and hence can be neglected for now without changing the inequality. Taking supremum over $[0,T]$ followed by taking the expectation of both sides and using the estimates established above gives (for some constant $C>0$)
			\begin{align}\label{eqn FG bounds lemma 2 inequality 5}
				\nonumber \mathbb{E} \sup_{t\in[0,T]} \l| \nabla m_n(t) \r|_{L^2}^2 \leq & C \l(  1 + \mathbb{E} \l| m_0 \r|_{H^1}^2  \r)
				+ C \mathbb{E} \int_{0}^{T} \l| m_n(s) \r|_{L^4}^4 \, ds + C \mathbb{E} \int_{0}^{T} \l| \int_{\mathbb{U}} \kappa(\cdot,u) \, q_s(du) \r|_{L^4}^4 \, ds \\
				& + C \mathbb{E} \int_{0}^{T} \l| m_n \r|_{L^2}^2 \, ds 
				+ C \mathbb{E} \int_{0}^{T} \l| \nabla m_n \r|_{L^2}^2 \, ds + \frac{\delta}{2} \mathbb{E}   \sup_{t\in[0,T]} \l| \nabla m_n(t) \r|_{L^2}^2 .
			\end{align}
			We choose $\delta$ small enough ($\delta = 1$ works here) so that the last term on the right hand side of the above inequality can be adjusted with the left hand side, still keeping the coefficient positive.
			By the inequalities \eqref{eqn FG L infinity L2 bound mn}, \eqref{eqn FG L4 L4 bound mn} in Lemma \ref{lemma FG bounds 1}, along with the Assumption \ref{assumption Main Assumption} (in particular on the initial data $m_0$ and the operator $L$), there exists a constant $C>0$ such that the following inequality holds.
			\begin{align}\label{eqn FG bounds lemma 2 inequality 6}
				\nonumber \mathbb{E} \sup_{t\in[0,T]} \l| \nabla m_n(t) \r|_{L^2}^2 \leq & C \l( 1 +  \mathbb{E} \int_{0}^{T} \l| \nabla m_n(s) \r|_{L^2}^2 \, ds \r) .
			\end{align}
			Applying the Gronwall inequality, and using the bound \eqref{eqn FG L infinity L2 bound mn} gives \eqref{eqn FG L infinity H1 bound mn}. The bound \eqref{eqn FG L2 H2 bound mn} can be obtained similar to the bound \eqref{eqn FG L2 H1 bound mn}, using the inequality \eqref{eqn FG bounds lemma 2 inequality 3}.
		\end{proof}
		
		\begin{lemma}\label{lemma FG bounds 3}
			There exists a constant $C>0$ such that for $\alpha\in(0,\frac{1}{2}),\ p\geq 2$, with $\alpha p > 1$, we have the following inequality.
			\begin{equation}\label{eqn FG W alpha p bound mn}
				\mathbb{E} \l| m_n \r|_{W^{\alpha,p}(0,T:(H^1)^\p)}^2 \leq C.
			\end{equation}
		\end{lemma}
		\begin{proof}[Proof of Lemma \ref{lemma FG bounds 3}]
			The proof can be given following \cite{ZB+BG+Le_SLLBE}, \cite{ZB+BG+TJ_Weak_3d_SLLGE}. We give arguments only for the control term.	It suffices to show that the control term is square integrable.
			\begin{align}
				\mathbb{E} \int_{0}^{T} \l| \int_{\mathbb{U}} L(m,v) \, q_s(dv) \r|_{L^2}^2 \, dt \leq C \l[ \mathbb{E} \int_{0}^{T} \l| m_n(s) \r|_{L^4}^4 \, ds \r]^{\frac{1}{2}} \l[ \mathbb{E} \int_{0}^{T} \int_{\mathbb{U}} \kappa^4(s,v) \, ds \r]^{\frac{1}{2}} < \infty.
			\end{align}
			Using the bound \eqref{eqn FG L4 L4 bound mn}, along with the Assumption \ref{assumption Main Assumption} gives the finiteness of the right hand side of the above inequality. Combining the above estimates with the compact embeddings \cite{Simon_Compact_Sets}
			\begin{equation}
				L^2 \hookrightarrow (H^1)^\p\ \text{and}\ W^{1,2}(0,T:(H^1)^\p) \dela{L^2)} \hookrightarrow W^{\alpha , p}(0,T:(H^1)^\p )
			\end{equation}
			
			gives the required regularity for the integral. For the stochastic integral, we refer the reader to Lemma 2.1 from \cite{Flandoli_Gatarek}.
		\end{proof}
		
		\subsection{Compactness}\label{section compactness}
		For the purpose of the section, let 
		$$X_T := L^2(0,T:H^1) \cap L^4(0,T:L^4) \cap  C([0,T]:(H^1)^\p).$$
		
		\begin{lemma}\label{lemma FG tightness}
			The sequence of laws of $\l( m_n , \lambda , W \r)$ is tight on the space\\
			$X_T \times \mathcal{Y}(0,T:\mathbb{U}) \times C([0,T]:\mathbb{R})$.
			\dela{
				We have the following.
				\begin{enumerate}
					\item Let $\mathcal{L}(m_n)$ denote the law of the process $m_n,n\in\mathbb{N}$. Then the sequence $\l\{ \mathcal{L}(m_n) \r\}_{n\in\mathbb{N}}$ is tight on the space $L^2(0,T:H^1) \cap L^4(0,T:L^4)\cap C([0,T] X^{-\beta})$.
					\item The measures $\lambda,W$ are tight on the spaces $\mathcal{Y}(0,T:\mathbb{U}),\ C([0,T]:\mathbb{R})$ respectively.
				\end{enumerate}
			}
		\end{lemma}
		\begin{proof}[Proof of Lemma \ref{lemma FG tightness}]
			The following embedding is compact \cite{Simon_Compact_Sets} for $\alpha,p$ as in Lemma \ref{lemma FG bounds 3} and $\alpha p >1$.
			$$Y_T := L^{\infty}(0,T:H^1) \cap L^{2}(0,T:H^2) \cap W^{\alpha , p}(0,T:(H^1)^\p) \hookrightarrow X_T.$$
			Using the bounds in Lemmas \ref{lemma FG bounds 1}, \ref{lemma FG bounds 2} and \ref{lemma FG bounds 3}, we can show that the laws of the processes $m_n,n\in\mathbb{N}$ are concentrated in a ball in the space $Y_T$. The compact embedding mentioned above concludes the tightness for the laws of $m_n,n\in\mathbb{N}$.
			Tightness for the law of $\lambda$ follows from Lemma \ref{lemma tightness} and the Assumption \ref{assumption Main Assumption}.
			Tightness for the law of $W$ follows from the fact that $W$ is a Wiener process and $C([0,T]:\mathbb{R})$ is a Radon space.
		\end{proof}
		
		\subsection{Use of Skorohod Theorem}\label{section Use of Skorohod Theorem}
		Using the Skorohod Theorem, we have the following lemma.
		\begin{lemma}\label{lemma use of Skorohod theorem}
			There exists a probability space $\l(\tilde{\Omega} , \tilde{\mathcal{F}} , \tilde{\mathbb{P}}\r)$, along with $X_T$-valued random variables $\tilde{m}_n,\tilde{m}, n\in\mathbb{N}$ on the said probability space, along with a sequence of random Young measures $\tilde{\lambda}_n, \tilde{\lambda},n\in\mathbb{N}$ on $\mathbb{U}$ and processes $\tilde{W}_n,\tilde{W}, n\in\mathbb{N}$ such that the following hold.
			
			\begin{equation}
				\l( \tilde{m}_n , \tilde{\lambda}_n , \tilde{W}_n \r) \stackrel{d}{=} \l( m_n , \lambda , W \r),\ \forall n\in\mathbb{N}.
			\end{equation}
			\begin{equation}
				\l( \tilde{m}_n , \tilde{\lambda}_n , \tilde{W}_n \r) \to \l( \tilde{m} , \tilde{\lambda} , \tilde{W} \r),\ \tilde{\mathbb{P}}\text{-a.s. in}\ X_T \times \mathcal{Y}(0,T:\mathbb{U}) \times C([0,T] : \mathbb{R}).
			\end{equation}

			\dela{
				
				\item $\tilde{m}_n$ and $\tilde{m}_n$ have the same laws on the space $L^2(0,T:H^1) \cap L^4(0,T:L^4) \cap  C([0,T] X^{-\beta})$, for each $n\in\mathbb{N}$.
				\item 
				\begin{equation}
					\tilde{m}_n \to \tilde{m} \text{ in } L^2(0,T:H^1) \cap L^4(0,T:L^4) \cap  C([0,T] X^{-\beta})\ \mathbb{P}^\p-a.s.
				\end{equation}
				\item $\tilde{\lambda}_n$ has the same law as that of $\lambda$ for each $n\in\mathbb{N}$ on the space of Young measures $\mathcal{Y}(0,T:\mathbb{U})$.
				\item 
				\begin{equation}
					\tilde{\lambda}_n \to \tilde{\lambda} \text{ stably in } \mathcal{Y}(0,T:\mathbb{U})\ \mathbb{P}^\p-a.s.
				\end{equation}
				\item $\tilde{W}_n$ has the same law as that of $W$ for each $n\in\mathbb{N}$ on the space of $C([0,T]:\mathbb{R})$.
				\item 
				\begin{equation}
					\tilde{W}_n \to \tilde{W} \text{ in } C([0,T]:\mathbb{R})\ \mathbb{P}^\p-a.s..
				\end{equation}

			}

		\end{lemma}
		
		By Lemma \ref{lemma disintegration of Young measures}, there exists a relaxed control process $\tilde{q}$ such that
		\begin{equation}
			\tilde{\lambda}(du,dt) = \tilde{q}_t(du) \, dt.
		\end{equation}
		Similar result holds for $\tilde{\lambda}_n$. That is, there exists there exists a relaxed control process $\tilde{q}^n$ such that
		\begin{equation}
			\tilde{\lambda}_n(du,dt) = \tilde{q}_t^n(du) \, dt,\ n\in\mathbb{N}.
		\end{equation}

		\begin{lemma}\label{lemma FG bounds mn tilde}
			There exists a constant $C>0$ such that the following hold.
			\begin{equation}
				\tilde{\mathbb{E}} \sup_{t\in[0,T]} \l| \tilde{m}_n(t) \r|_{H^1}^2  \leq C,
			\end{equation}
			
			\begin{equation}
				\tilde{\mathbb{E}} \int_{0}^{T} \l| \tilde{m}_n(t) \r|_{H^2}^2 \, dt \leq C,
			\end{equation}
			
			\begin{equation}
				\tilde{\mathbb{E}} \int_{0}^{T} \l| \tilde{m}_n(t) \r|_{L^4}^4 \, dt \leq C.
			\end{equation}
			
		\end{lemma}
		\begin{proof}
			First, we observe by Kuratowski's Theorem (see Theorem 1.1, \cite{Vakhania_Probability_distributions_on_Banach_spaces}) that the Borel subsets of $C([0,T]:H_n)$ are also Borel subsets of the space $L^4(0,T:L^4) \cap L^2(0,T:H^1) \cap C([0,T]: (H^1)^\p)$. Since $m_n$ and $\tilde{m}_n$ have the same laws on the space $C([0,T]:H_n)$ we can show that the sequence $\tilde{m}_n$ satisfies the same bounds as that satisfied by the sequence $m_n$ (Lemmas \ref{lemma FG bounds 1}, \ref{lemma FG bounds 2}), for each $n\in\mathbb{N}$, thus completing the proof of the lemma.
			
		\end{proof}
		Lemma \ref{lemma FG bounds mn tilde}, along with the convergence in Lemma \ref{lemma use of Skorohod theorem} and some lower semicontinuity results gives us the following lemma (see \cite{ZB+BG+Le_SLLBE}, \cite{ZB+BG+TJ_Weak_3d_SLLGE} for instance).
		\begin{lemma}\label{lemma FG bounds m tilde}
			There exists a constant $C>0$ such that the following hold.
			\begin{equation}
				\tilde{\mathbb{E}} \sup_{t\in[0,T]} \l| \tilde{m}(t) \r|_{H^1}^2  \leq C,
			\end{equation}
			
			\begin{equation}
				\tilde{\mathbb{E}} \int_{0}^{T} \l| \tilde{m}(t) \r|_{H^2}^2 \, dt \leq C,
			\end{equation}
			
			\begin{equation}
				\tilde{\mathbb{E}} \int_{0}^{T} \l| \tilde{m}(t) \r|_{L^4}^4 \, dt \leq C.
			\end{equation}
		\end{lemma}

		\subsection{Convergence}\label{section convergence}
		The subsection aims to show that the obtained limit $\tilde{m}$ is a solution to the equation \eqref{eqn problem considered relaxed weak form with Ito integral} corresponding to the Young measure $\tilde{\lambda}$ and Wiener process $\tilde{W}$ on the probability space $\l(\tilde{\Omega} , \tilde{\mathcal{F}} , \tilde{\mathbb{P}}\r)$. For the sake of simplicity of writing, we replace the notations $\tilde{m}_n,\tilde{m},\tilde{\lambda}_n,\tilde{\lambda}$, $\tilde{W}_n, \tilde{W}$ respectively by $m_n,m,\lambda_n,\lambda,W_n, W$. Similar notation changes are to be noted for the probability space. Note that the said processes are not to be confused with the finite dimensional approximations in Section \ref{section FG uniform energy estimates}.
		
		We give arguments only for the control term. Rest follow from, for example, Section 5, \cite{ZB+BG+Le_SLLBE} (see also \cite{ZB+BG+TJ_Weak_3d_SLLGE}).
		Towards that, it suffices to have the following convergence.
		\begin{lemma}\label{lemma FG convergence of control term}
			Let $\phi\in L^4(\Omega:H^1)$.
			\begin{equation}
				\lim_{n\to\infty} \mathbb{E} \l| \int_{0}^{T} \int_{\mathbb{U}} \l\langle L(m_n,u) , \phi \r\rangle_{L^2} \, \lambda_n(du,dt) - \int_{0}^{T} \int_{\mathbb{U}} \l\langle L(m,u), \phi \r\rangle_{L^2} \, \lambda(du,dt) \r|_{L^2}^2 = 0.
			\end{equation}
		\end{lemma}
		\begin{proof}[Proof of Lemma \ref{lemma FG convergence of control term}]
			First, we observe the following.
			\begin{align}
				\nonumber & \mathbb{E} \l| \int_{0}^{T} \int_{\mathbb{U}} \l\langle L(m_n,v) , \phi \r\rangle_{L^2} \, \lambda_n(dv,dt) - \int_{0}^{T} \int_{\mathbb{U}} \l\langle L(m,v) , \phi \r\rangle_{L^2} \, \lambda(dv,dt) \r| \\
				\leq & 
				\nonumber \mathbb{E} \l| \int_{0}^{T} \int_{\mathbb{U}} \l\langle L(m_n,v) , \phi \r\rangle_{L^2} \, \lambda_n(dv,dt) 
				- \int_{0}^{T} \int_{\mathbb{U}} \l\langle L(m,v) , \phi \r\rangle_{L^2} \, \lambda_n(dv,dt) \r| \ \l( \text{Term} 1 \r) \\
				\nonumber & + \mathbb{E} \l| \int_{0}^{T} \int_{\mathbb{U}} \l\langle L(m,v) , \phi \r\rangle_{L^2} \, \lambda_n(dv,dt) 
				- \int_{0}^{T} \int_{\mathbb{U}} \l\langle L(m,v) , \phi \r\rangle_{L^2} \, \lambda(dv,dt) \r|. \ \l( \text{Term 2} \r)
			\end{align}		
			\textbf{Calculations for Term 1}:\\		
			Let $R>0$. Let us define a cut-off function $\psi_R$ as follows. Let $v\in\mathbb{U}$. 
			\begin{align}\label{eqn definition of Psi R 1}
				\psi_R(v) = \begin{cases}
					1,&\ \text{if}\ \kappa(\cdot,v) \leq R, \\
					0,&\ \text{if}\ \kappa(\cdot,v) \geq 2R.
				\end{cases}
			\end{align}
			Then we have
			\begin{align}
				\nonumber & \mathbb{E} \l| \int_{0}^{T} \int_{\mathbb{U}} \l\langle L(m_n,v) , \phi \r\rangle_{L^2} \, \lambda_n(dv,dt) 
				- \int_{0}^{T} \int_{\mathbb{U}} \l\langle L(m,v) , \phi \r\rangle_{L^2} \, \lambda_n(dv,dt) \r| \\
				\nonumber  = &  \mathbb{E} \l| \int_{0}^{T} \int_{\mathbb{U}} \l\langle \l( 1 - \psi_R(v) \r) L(m_n,v) , \phi \r\rangle_{L^2} \, \lambda_n(du,dt) \r|
				+  \mathbb{E} \l| \int_{0}^{T} \int_{\mathbb{U}}  \l( 1 - \psi_R(v) \r) \l\langle L(m,v) , \phi \r\rangle_{L^2} \, \lambda_n(dv,dt) \r| \\
				& +   \mathbb{E} \l| \int_{0}^{T} \int_{\mathbb{U}} \l\langle \psi_R(u) L(m_n,v) , \phi \r\rangle_{L^2} \, \lambda_n(dv,dt) 
				- \int_{0}^{T} \int_{\mathbb{U}} \psi_R(v) \l\langle L(m,v) , \phi \r\rangle_{L^2} \, \lambda_n(dv,dt) \r| .
			\end{align}
			For the first of the above three terms, we have the following.
			Let $K_R$ denote the level set given by
			$$K_R = \l\{ v : \kappa(\cdot,v) \geq R \r\}.$$
			
			\begin{align}
				\nonumber  \mathbb{E} \l| \int_{0}^{T} \int_{\mathbb{U}} \l\langle \l( 1 - \psi_R(v) \r) L(m_n,v) , \phi \r\rangle_{L^2} \, \lambda_n(dv,dt) \r| 
				\leq &   \mathbb{E} \int_{0}^{T} \int_{K_R}  \l| L(m_n,v) \r|_{L^2} \l| \phi \r|_{L^2} \, \lambda_n(dv,dt) \\
				\nonumber \leq &  \mathbb{E} \int_{0}^{T} \int_{K_R} \kappa(t,v) \l|m_n(t) \r|_{L^2} \l| \phi \r|_{L^2} \, \lambda_n(dv,dt) \\
				\nonumber \leq & \frac{1}{R}  \mathbb{E} \int_{0}^{T} \int_{K_R}  \kappa^2(t,v) \l|m_n(t) \r|_{L^2} \l| \phi \r|_{L^2} \, \lambda_n(dv,dt) \\
				\leq & \frac{C}{R}.
			\end{align}	
			The right hand side of the above inequality goes to $0$ as $n$ to infinity. The second term can be handled similarly.\\			
			We now show the calculations for the last term. We recall the term for readers convenience.
			\begin{equation}
				\mathbb{E} \l| \int_{0}^{T} \int_{\mathbb{U}} \l\langle \psi_R(v) L(m_n,v) , \phi \r\rangle_{L^2} \, \lambda_n(dv,dt) 
				- \int_{0}^{T} \int_{\mathbb{U}} \psi_R(v) \l\langle L(m,v) , \phi \r\rangle_{L^2} \, \lambda_n(dv,dt) \r| .
			\end{equation}
			
			First, we note the following.
			\begin{align}
				\nonumber &  \mathbb{E} \l| \int_{0}^{T} \int_{\mathbb{U}} \l\langle \psi_R(v) L(m_n,v) , \phi \r\rangle_{L^2} \, \lambda_n(dv,dt) 
				- \int_{0}^{T} \int_{\mathbb{U}} \psi_R(v) \l\langle L(m,v) , \phi \r\rangle_{L^2} \, \lambda_n(dv,dt) \r| \\
				\nonumber \leq &  \mathbb{E} \int_{0}^{T} \l[ \sup_{v\in K_R} \l| L(m_n,v) - L(m,v) \r|_{L^2} \int_{K_R}    \l| \phi \r|_{L^2} \, \r]\lambda_n(dv,dt) .
			\end{align}
			\dela{Note that by Assumption \ref{assumption Main Assumption}, we have
				\begin{equation}
					\lim_{n\to\infty}  \sup_{u\in K_R} \l| L(m_n,u) - L(m,u) \r|_{L^2} = 0.
			\end{equation}}
			Moreover,
			\begin{align}
				\nonumber & \mathbb{E} \int_{0}^{T} \sup_{v\in K_R} \l| L(m_n,v) - L(m,v) \r|_{L^2}^2 \int_{K_R}    \l| \phi \r|_{L^2}^2 \, \lambda_n(dv,dt) \\
				\leq & C(R) \l( \mathbb{E} \int_{0}^{T}  \bigl[ \l| m_n(t)\r|_{L^2}^4 + \l| m(t)\r|_{L^2}^4 \bigr]  \, dt \r)^{\frac{1}{2}} 
				\l( \mathbb{E} \int_{0}^{T}   \l| \phi \r|_{L^2}^2 \, dt \r)^{\frac{1}{2}} < \infty.
			\end{align}
			Therefore by Vitali Convergence Theorem, combined with Assumption \ref{assumption Main Assumption}, we have the following for each $R>0$.
			\begin{equation}
				\lim_{n\to\infty}  \mathbb{E} \l| \int_{0}^{T} \int_{\mathbb{U}} \l\langle \psi_R(v) L(m_n,v) , \phi \r\rangle_{L^2} \, \lambda_n(dv,dt) 
				- \int_{0}^{T} \int_{\mathbb{U}} \psi_R(v) \l\langle L(m,v) , \phi \r\rangle_{L^2} \, \lambda_n(dv,dt) \r| = 0.
			\end{equation}
			Summing up, we have the following.
			
			\begin{equation}
				\lim_{n\to\infty} \mathbb{E} \l| \int_{0}^{T} \int_{\mathbb{U}} \l\langle L(m_n,v) , \phi \r\rangle_{L^2} \, \lambda_n(dv,dt) 
				- \int_{0}^{T} \int_{\mathbb{U}} \l\langle L(m,v) , \phi \r\rangle_{L^2} \, \lambda_n(dv,dt) \r| = 0.
			\end{equation}		
			This concludes the calculations for Term 1.\\	
			\textbf{Calculations for Term 2:} We define another cut-off function as done before. Let $v\in\mathbb{U}$. Let 
			\begin{align}\label{eqn definition of Psi R 2}
				\Psi_R(m,v) = \begin{cases}
					1,&\ \text{if}\ \l| m \r|_{L^2} \leq R \text{ and }  \kappa(\cdot,v)  \leq R,\\
					0,&\ \text{if}\ \l| m \r|_{L^2} \geq 2R \ \text{ or } \kappa(\cdot,v) \geq 2R.
				\end{cases}
			\end{align}
			
			\begin{align}
				\nonumber & \mathbb{E} \l| \int_{0}^{T} \int_{\mathbb{U}} \l\langle L(m,v) , \phi \r\rangle_{L^2} \, \lambda_n(dv,dt) 
				- \int_{0}^{T} \int_{\mathbb{U}} \l\langle L(m,v) , \phi \r\rangle_{L^2} \, \lambda(dv,dt) \r| \\
				\nonumber \leq & \mathbb{E} \bigg| \int_{0}^{T} \int_{\mathbb{U}} \l( 1 - \Psi_R(m(t),v) \r) \l\langle L(m,v) , \phi \r\rangle_{L^2} \, \lambda_n(dv,dt) \bigg| \\
				\nonumber & + \bigg| \int_{0}^{T} \int_{\mathbb{U}}  \l( 1 - \Psi_R(m(t),v) \r) \l\langle L(m,v) , \phi \r\rangle_{L^2} \, \lambda(dv,dt) \bigg| \\
				\nonumber & + \mathbb{E} \bigg| \int_{0}^{T} \int_{\mathbb{U}} \Psi_R(m(t),v)\l\langle L(m,v) , \phi \r\rangle_{L^2} \, \lambda_n(dv,dt)  - \int_{0}^{T} \int_{\mathbb{U}} \Psi_R(m(t),v)\l\langle L(m,v) , \phi \r\rangle_{L^2} \, \lambda(dv,dt) \bigg|.
			\end{align}
			The first two terms can be handled as done previously. For the third term, we have the following.

			First, we observe that by Assumption \ref{assumption Main Assumption} and the definition of $\Psi_R$ (equation \eqref{eqn definition of Psi R 2}), we have $\Psi_R(m(\cdot)) L(m(\cdot)) \in L^1(0,T:C_{\text{b}}(\mathbb{U}))$.		
			Now, by Lemma \ref{lemma for convergence of integral of continuous bounded function Young measures}, we have the following.
			\begin{align}
				\nonumber \lim_{n\to\infty} \mathbb{E}  \bigg| \int_{0}^{T} \int_{\mathbb{U}} \Psi_R(m,v)\l\langle L(m,v) , \phi \r\rangle_{L^2} \, \lambda_n(dv,dt)  - \int_{0}^{T} \int_{\mathbb{U}} \Psi_R(m(t),v)\l\langle L(m,v) , \phi \r\rangle_{L^2} & \, \lambda(dv,dt) \bigg| \\
				& = 0.
			\end{align}
			Summing up, we have
			\begin{align}
				\lim_{n\to\infty} \mathbb{E} \l| \int_{0}^{T} \int_{\mathbb{U}} \l\langle L(m,v) , \phi \r\rangle_{L^2} \, \lambda_n(dv,dt) 
				- \int_{0}^{T} \int_{\mathbb{U}} \l\langle L(m,v) , \phi \r\rangle_{L^2} \, \lambda(dv,dt) \r| = 0.
			\end{align}
			This concludes the calculations for Term 2 and hence also the proof of the lemma.
		\end{proof}

		We now go back to the notations set up in Section \ref{section Use of Skorohod Theorem}. Using the convergences in Lemma \ref{lemma FG convergence of control term} and following Section 5 in \cite{ZB+BG+Le_SLLBE} (see also \cite{ZB+BG+TJ_Weak_3d_SLLGE}), we can show that the process $\tilde{m}$ satisfies the equality \eqref{eqn problem considered relaxed weak form with Ito integral} corresponding to the Wiener process $\tilde{W}$ and the random Young measure $\tilde{\lambda}$. Therefore, we can conclude that the tuple $\tilde{\pi} = \l( \tilde{\Omega} , \tilde{\mathcal{F}} , \tilde{\mathbb{F}} , \tilde{\mathbb{P}} , \tilde{m} , \tilde{\lambda} , \tilde{W} \r)$ is a weak martingale solution to the relaxed problem \eqref{eqn problem considered relaxed}, thus concluding the proof of Theorem \ref{theorem existence of weak martingale solution}.

		\section{Existence of a Optimal Control}\label{section existence of a optimal control}
		\begin{proof}[Idea of proof of Theorem \ref{theorem existence of weak relaxed optimal control}]
			The section aims to show that the problem \eqref{eqn problem considered original} admits a weak relaxed control as in Definition \ref{definition weak relaxed optimal control}. We give a sketch of the proof.
			In Step 1, we claim that the space $\mathcal{U}(m_0,T)$ of admissible solutions is non-empty, and hence the infimum $\Lambda$ is finite.
			Therefore, there exists a minimizing sequence of weak martingale solutions. From this sequence, we extract a subsequence that converges to another weak martingale solution. We then use the lower semicontinuity of the costs $F,\Psi$ to show that the infimum $\Lambda$ is attained at the so obtained limit.
		\end{proof}
		
		\begin{enumerate}
			\item[\textbf{Step 1:}] 
			Theorem \ref{theorem existence of weak martingale solution}\dela{ (also Theorem \ref{theorem existence strong solution})} implies that the problem \eqref{eqn problem considered relaxed} admits a weak martingale solution. Hence, we can say that the space of admissible solutions $\mathcal{U}(m_0,T)$ is non-empty, and hence $\Lambda$ is finite. 
			Suppose not. That is, if every weak martingale solution has infinite cost, then the minimum of costs is also infinite, and hence it is attained (since we have shown that at least one solution exists).
			Therefore it suffices to study the case where at least one solution has finite cost.
			
			\item[\textbf{Step 2:}] Since $\Lambda<\infty$, there exists a minimizing sequence $\l\{ \pi_n \r\}_{n\in\mathbb{N}} = \l( \Omega_n , \mathcal{F}_n , \mathbb{F}_n , \mathbb{P}_n , m_n , \lambda_n , W_n \r)$ of admissible solutions.
			In particular, each $m_n$ satisfies the following equation for $n\in\mathbb{N}$.
			
			\begin{align}\label{eqn oc eqn for mn}
				\nonumber m_n(t) = & m_0 + \int_{0}^{t} \Delta m_n(s) + m_n(s) \times \Delta m_n(s) - \l( 1 + \l| m_n(s) \r|_{\mathbb{R}^3}^2 \r) \, ds \\
				\nonumber & + \int_{0}^{t} \int_{\mathbb{U}} \l[ L\bigl(m_n(s),v\bigr) \r] \, \lambda_n(dv,ds) \\
				& + \frac{1}{2} \int_{0}^{t} \l( m_n(s) \times h \r) \times h \, ds
				+ \int_{0}^{t} \l( m_n(s) \times h + h \r)\,  dW_n(s),\ t\in[0,T].
			\end{align}
			
			Since the said sequence is a minimizing one, we have the existence of a constant $R>0$ such that for each $n\in\mathbb{N}$,
			\begin{equation}
				J(\pi_n) \leq R.
			\end{equation}
			By the coercivity condition \eqref{eqn coercivity assumption}, there exists a constant $C>0$ independent of $n$ such that
			\begin{align}
				\mathbb{E}^n \int_{0}^{T} \kappa^4 (t,v) \lambda_n(dv,dt) \leq C \mathbb{E}^n \int_{0}^{T} F\big(m(t),v\big) \lambda_n(dv,dt) \leq J(\pi_n) \leq R.
			\end{align}
			Here $\mathbb{E}^n$ denotes the expectation with respect to the probability space $\l(\Omega_n, \mathcal{F}_n, \mathbb{P}_n\r)$.
			
			Concerning uniform energy estimates, first, we observe that in Lemmata \ref{lemma FG bounds 1}, \ref{lemma FG bounds 2}, \ref{lemma FG bounds 3}, and hence also Lemma \ref{lemma FG bounds m tilde} the bounds depend on the bound on the inf compact function $\kappa$. Hence the uniform bounds (independent of $n\in\mathbb{N}$) can be established for the aforementioned approximating sequence. This is made rigorous in the following lemma.
			
			\begin{lemma}\label{lemma oc bounds lemma 1}
				There exists a constant $C>0$, which is independent of $n$ such that the following hold.
				\begin{equation}\label{eqn oc L infinity H1 bound mn}
					\mathbb{E}^n\sup_{t\in[0,T]} \l| m_n(t) \r|_{H^1}^2 \leq C,
				\end{equation}
				\begin{equation}\label{eqn oc L2 H2 bound mn}
					\mathbb{E}^n\int_{0}^{T} \l| m_n(t) \r|_{H^2}^2 \, dt \leq C.
				\end{equation}
			\end{lemma}			
			A result of $m_n$ satisfying the equation \eqref{eqn oc eqn for mn} and the above Lemma \ref{lemma oc bounds lemma 1} is the following lemma.
			\begin{lemma}\label{lemma oc bounds 2}
				There exists a constant $C>0$ such that for $\alpha\in(0,\frac{1}{2})$, $p\geq 2$, with $\alpha p>1$, we have the following inequality.
				\begin{equation}\label{eqn oc W alpha p bound mn}
					\mathbb{E}^n \l| m_n \r|_{W^{\alpha,p}(0,T:(H^1)^\p)}^2 \leq C.
				\end{equation}
			\end{lemma}
			
			\item[\textbf{Step 3:}] As done for the existence part (see Lemma \ref{lemma FG tightness} in particular), we show that the sequence of laws of $ \l( m_n, \lambda_n, W_n \r) ,n\in\mathbb{N}$ is tight on some space. This is made rigorous in the following lemma. We again use a notation used in Section \ref{section compactness}.\\
			Let $X_T = L^2(0,T:H^1) \cap L^4(0,T:L^4) \cap  C([0,T]:(H^1)^\p)$.
			\begin{lemma}\label{lemma oc tightness}		
				The sequence of laws of $\l( m_n , \lambda_n , W_n \r)$ is tight on the space $X_T \times \mathcal{Y}(0,T:\mathbb{U}) \times C([0,T]:\mathbb{R})$.
				
					%

			\end{lemma}		
			Using the Skorohod Theorem, we have the following lemma.
			\begin{lemma}\label{lemma oc use of Skorohod theorem}
				There exists a probability space $\l(\tilde{\Omega} , \tilde{\mathcal{F}} , \tilde{\mathbb{P}}\r)$, along with 
				$X_T$-valued random variables $\l\{\tilde{m}_n\r\}_{n\in\mathbb{N}},\tilde{m}$ on the said probability space, along with a sequence of random Young measures $\l\{\tilde{\lambda}_n\r\}_{n\in\mathbb{N}}, \tilde{\lambda}$ on $\mathbb{U}$ and processes $\l\{\tilde{W}_n\r\}_{n\in\mathbb{N}},\tilde{W}$ such that the following hold.
				\begin{equation}
					\l( \tilde{m}_n , \tilde{\lambda}_n , \tilde{W}_n \r) \stackrel{d}{=} \l( m_n , \lambda_n , W_n \r),\ \forall n\in\mathbb{N}.
				\end{equation}
				\begin{equation}
					\l( \tilde{m}_n , \tilde{\lambda}_n , \tilde{W}_n \r) \to \l( \tilde{m} , \tilde{\lambda} , \tilde{W} \r),\ \tilde{\mathbb{P}}\text{-a.s. in}\ X_T \times \mathcal{Y}(0,T:\mathbb{U}) \times C([0,T] : \mathbb{R}).
				\end{equation}
				
				\dela{
					\item $\tilde{m}_n$ and $\tilde{m}_n$ have the same laws on the space $L^2(0,T:H^1) \cap L^4(0,T:L^4) \cap  C([0,T] : (H^1)^\p)$, for each $n\in\mathbb{N}$.
					\item 
					\begin{equation}
						\tilde{m}_n \to \tilde{m} \text{ in } L^2(0,T:H^1) \cap L^4(0,T:L^4) \cap  C([0,T] : (H^1)^\p)\ \mathbb{P}^\p-a.s.
					\end{equation}
					\item $\tilde{\lambda}_n$ has the same law as that of $\lambda_n$ for each $n\in\mathbb{N}$ on the space of Young measures $\mathcal{Y}(0,T:\mathbb{U})$.
					\item 
					\begin{equation}
						\tilde{\lambda}_n \to \tilde{\lambda} \text{ stably in } \mathcal{Y}(0,T:\mathbb{U})\ \mathbb{P}^\p-a.s..
					\end{equation}
					\item $\tilde{W}_n$ has the same law as that of $W_n$ for each $n\in\mathbb{N}$ on the space of $C([0,T]:\mathbb{R})$.
					\item 
					\begin{equation}
						\tilde{W}_n \to \tilde{W} \text{ in } C([0,T]:\mathbb{R})\ \mathbb{P}^\p-a.s..
					\end{equation}

				}
			\end{lemma}
			By Kuratowski's Theorem, we can show that the processes $\tilde{m}_n$ satisfy the same bounds that are satisfied by the processes $m_n$, for each $n\in\mathbb{N}$. This is given in the following lemma.
			
			\begin{lemma}\label{lemma oc bounds lemma 3}
				There exists a constant $C>0$, which is independent of $n$ such that the following hold.
				\begin{equation}
					\tilde{\mathbb{E}}\sup_{t\in[0,T]} \l| \tilde{m}_n(t) \r|_{H^1}^2 \leq C,
				\end{equation}
				\begin{equation}
					\tilde{\mathbb{E}}\int_{0}^{T} \l| \tilde{m}_n(t) \r|_{H^2}^2 \, dt \leq C.
				\end{equation}
			\end{lemma}
			Similar to Lemma \ref{lemma FG bounds m tilde}, we can show that following lemma.
			\begin{lemma}\label{lemma oc bounds lemma 4}
				There exists a constant $C>0$, such that the following hold.
				\begin{equation}\label{eqn oc L infinity H1 bound m}
					\tilde{\mathbb{E}}\sup_{t\in[0,T]} \l| \tilde{m}(t) \r|_{H^1}^2 \leq C,
				\end{equation}
				
				\begin{equation}\label{eqn oc L2 H2 bound m}
					\tilde{\mathbb{E}}\int_{0}^{T} \l| \tilde{m}(t) \r|_{H^2}^2 \, dt \leq C.
				\end{equation}
			\end{lemma}

			Lemma \ref{lemma oc use of Skorohod theorem} gives pointwise $(\tilde{\mathbb{P}}$-a.s.) convergence of $\tilde{m}_n$ to $\tilde{m}$. Following the arguments in Section \ref{section existence of a solution for the relaxed control}, we can show that the obtained tuple $\tilde{\pi} = \l( \tilde{\Omega} , \tilde{\mathcal{F}} , \tilde{\mathbb{F}} , \tilde{\mathbb{P}} , \tilde{m} , \tilde{\lambda} , \tilde{W} \r)$ is a weak martingale solution to the problem \eqref{eqn problem considered relaxed}.
			
			\item[\textbf{Step 4:}] It now remains to show that the obtained solution indeed minimizes the cost.
			We have the following convergences $\tilde{\mathbb{P}}$-a.s.
				\begin{align}
					\tilde{m}_n \to \tilde{m}\ \text{to}\ L^2(0,T:H^1),\ \tilde{\lambda}_n \to \tilde{\lambda}\ \text{in}\ \mathcal{Y}(0,T:\mathbb{U}).
				\end{align}
			We recall (Assumption \ref{assumption assumption on cost functional}) that $F\geq 0$. Following the proof of Lemma 10 in \cite{Sritharan_DeterministicAndStochasticControl_SNSE} (see also, Lemma A.14 in \cite{UM+DM}, Lemma 2.15 in \cite{ZB+RS}, \cite{Haussman+Lepeltier_1990_OnExistenceOptimalControls}, \cite{Jean+Jean_1981Book}), we can show that the relaxed cost functional is lower semicontinuous. In particular we have the following inequality.
			\begin{align}
				\int_{0}^{T} \int_{\mathbb{U}} F(t,\tilde{m} , v) \, \tilde{\lambda}(dv , dt) \leq \liminf_{n\to\infty} \int_{0}^{T} \int_{\mathbb{U}} F(t,\tilde{m}_n , v ) \, \tilde{\lambda}_n (dv , dt).
			\end{align}
			
			Since the laws of $\tilde{m}_n$, $\tilde{\lambda}_n$ are the same as those of $m_n$, $\lambda_n$, we have
			\begin{align}\label{eqn oc minimum const inequality 1}
				\liminf_{n\to\infty} \tilde{\mathbb{E}} \int_{0}^{T} \int_{\mathbb{U}} F\l(\tilde{m}_n(t) , v\r) \, \tilde{\lambda}_n(dv,dt) \leq \liminf_{n\to\infty} \mathbb{E} \int_{0}^{T} \int_{\mathbb{U}} F\l(m_n(t) , v\r) \, \lambda_n(dv,dt).
			\end{align}
			By the above inequality and Fatou's Lemma, the continuity of $\Psi$ along with the fact that $\pi_n$ is a minimizing sequence, we have
			
			\begin{align}\label{eqn oc minimum const inequality 2}
				\nonumber \mathcal{J}(\tilde{\pi}) = \tilde{\mathbb{E}} \l[ \int_{0}^{T} \int_{\mathbb{U}} F\l(\tilde{m}(t) , v \r) \, \tilde{\lambda}(dv,dt) + \Psi(\tilde{m}(T)) \r] & \leq \liminf_{n\to\infty} \tilde{\mathbb{E}} \l[ \int_{0}^{T} F\l(\tilde{m}_n(t) , v \r) \, \tilde{\lambda}_n(dv,dt) + \Psi(\tilde{m}_n(T)) \r] \\
				\nonumber & \leq \liminf_{n\to\infty} \mathbb{E} \l[ \int_{0}^{T} F\l(m_n(t) , v \r) \, \lambda_n(dv,dt) + \Psi(m_n(T)) \r]  \\
				& = \Lambda.
			\end{align}
			Hence combining the above two estimates (\eqref{eqn oc minimum const inequality 1} and \eqref{eqn oc minimum const inequality 2}) gives
			\begin{align}
				\mathcal{J}(\tilde{\pi}) = \tilde{\mathbb{E}} \l[ \int_{0}^{T} F\l(\tilde{m}(t) , v \r) \, \tilde{\lambda}(dv,dt) + \Psi(\tilde{m}(T)) \r] = \Lambda.
			\end{align}
			Hence, the tuple $\tilde{\pi}$ minimizes the cost and therefore is a weak relaxed optimal control, as desired.
		\end{enumerate}
		
		\textbf{Disclosure Statement:} The authors report there are no competing interests to declare.
		
		\appendix
		
		\section{Pathwise Uniqueness and Existence of a Strong Solution}\label{section pathwise uniqueness and existence of a strong solution}
		The section focuses on proving that the problem \eqref{eqn problem considered relaxed} admits a pathwise unique solution. We reiterate that, in this section, we restrict ourselves to $d = 1,2$.
		\begin{theorem}\label{theorem pathwise uniqueness of solution}
			Let Assumption \ref{assumption Main Assumption} hold. Recall that the last assumption in Assumption \ref{assumption Main Assumption} will now be considered. Then the solution to \eqref{eqn problem considered relaxed} is pathwise unique.
			That is, the following holds. Let $m_1,m_2$ denote two weak martingale solutions to the problem \eqref{eqn problem considered relaxed}, corresponding to the same Wiener process and Young measure, on the same probability space, with the same initial data $m_0$. Then for each $t\in[0,T]$,
			\begin{equation}
				m_1(t) - m_2(t) = 0,\ \mathbb{P}-a.s.
			\end{equation}
		\end{theorem}
		Combining the pathwise uniqueness obtained above with the theory of Yamada and Watanabe \cite{Ikeda+Watanabe} gives us the following result.
		\begin{theorem}\label{theorem existence strong solution}
			The relaxed problem \eqref{eqn problem considered relaxed} admits a strong solution.
		\end{theorem}

		\begin{proof}[Idea of the proof of Theorem \ref{theorem pathwise uniqueness of solution}]
			In order to start the proof of Theorem \ref{theorem pathwise uniqueness of solution}, we first show (Corollary \ref{corollary equality in H1 prime}) that the equality \eqref{eqn problem considered relaxed} makes sense in $\l( H^1 \r)^\p$.  The proof then is standard. We consider $m_1,m_2$ as in Theorem \ref{theorem pathwise uniqueness of solution}. The equation satisfied by their difference $m$ is then considered. An application of the It\^o formula, followed by simplification, energy estimates and the use of Gronwall's inequality gives the required uniqueness. The structure of the proof is similar to the proof of Theorem 2.3 in \cite{ZB+BG+Le_SLLBE}. We show the calculations for the control term and refer the reader to, for instance, Section 2 in \cite{ZB+BG+Le_SLLBE}, for more details about the remaining calculations.
		\end{proof}
		\begin{corollary}\label{corollary equality in H1 prime}
			The equality \eqref{eqn problem considered relaxed} makes sense in $\l( H^1 \r)^\p$.
		\end{corollary}
		\begin{proof}[Proof of Corollary \ref{corollary equality in H1 prime}]
			The proof can be given by showing that each of the integrals on the right hand side of the equation \eqref{eqn problem considered relaxed} makes sense in the space $\l( H^1 \r)^\p$. We give estimates only for the control term. By Assumption \ref{assumption Main Assumption}, there exists a constant $C>0$ such that
			\begin{align}
				\l| \int_{0}^{t} \int_{\mathbb{U}} L(m(s) , v) \, q_s(dv) \, ds \r|_{\l(H^1\r)^\p}^2 
				\leq & C \l( \int_{0}^{t} \l| m(s) \r|_{L^2}^4 \, ds \r)^{\frac{1}{2}} \l( \int_{0}^{t} \int_{\mathbb{U}} \kappa^2(s,v) \, q_s(dv)  \, ds \r)^{\frac{1}{2}} < \infty.
			\end{align}
			This gives the required regularity for the control term, hence concluding the result.
		\end{proof}

		We now proceed to give a brief structure for the proof of Theorem \ref{theorem pathwise uniqueness of solution}.
		\begin{proof}[Proof of Theorem \ref{theorem pathwise uniqueness of solution}]
			For $i=1,2$, the solution $m_i$ satisfies the following equation.
			\begin{align}
				\nonumber m_i(t) = &  m_0 + \int_{0}^{t} \Delta m_i(s) + m_i(s) \times \Delta m_i(s) - \l( 1 + \l| m_i(s) \r|_{\mathbb{R}^3}^2 \r) m_i(s) \, ds \\
				\nonumber & + \int_{0}^{t} \int_{\mathbb{U}} \bigl[ L(m_1(s),v) + L(m_1(s),v) \bigr] \, q_s(dv) \, ds 
				+ \frac{1}{2} \int_{0}^{t} \bigl( m_i(s) \times h \bigr) \times h \, ds\\
				&+ \int_{0}^{t} \bigl( m_i(s) \times h + h \bigr) \, dW(s).
			\end{align}
			
			Let $m(t) = m_1(t) - m_2(t)$.
			Then the process $m$ satisfies the following equation.
			
			\begin{align}\label{eqn pathwise uniqueness equality for m}
				\nonumber m(t) = & \int_{0}^{t} \Delta m(s) \, ds + \int_{0}^{t} \bigl( m(s) \times \Delta m_1(s) - m_2(s) \times \Delta m_2(s) \bigr) \, ds \\
				\nonumber &  \int_{0}^{t} \bigl( m_1(s) \times \Delta m_1(s) - m_2(s) \times \Delta m(s) \bigr) \, ds \\
				\nonumber & - \int_{0}^{t} \l( \l( 1 + \l| m_1(s) \r|_{\mathbb{R}^3}^2 \r) m_1(s) - \l( 1 + \l| m_1(s) \r|_{\mathbb{R}^3}^2 \r) m_2(s) \r) \, ds \\
				\nonumber & - \int_{0}^{t} \l( \l( 1 + \l| m_1(s) \r|_{\mathbb{R}^3}^2 \r) m_2(s) - \l( 1 + \l| m_2(s) \r|_{\mathbb{R}^3}^2 \r) m_2(s) \r) \, ds \\
				\nonumber & +  \int_{0}^{t} \int_{\mathbb{U}} \bigl[  L(m_1(s),v) - L(m_2(s),v) \bigr] \, q_s(dv) \, ds \\
				& + \frac{1}{2} \int_{0}^{t} \bigl( m(s) \times h \bigr) \times h \, ds + \int_{0}^{t} \bigl( m(s) \times h \bigr) \, dW(s).
			\end{align}
			Applying the It\^o formula gives
			
			\begin{align}\label{eqn pathwise uniqueness equality for m Ito formula}
				\nonumber \frac{1}{2} \l| m(t) \r|_{L^2}^2 = & 
				- \int_{0}^{t} \l| \nabla m(s) \r|_{L^2}^2 \, ds 
				+ \int_{0}^{t} \l\langle m(s) \times \Delta m_1(s) , m(s) \r\rangle_{L^2} \, ds\\
				\nonumber &  - \int_{0}^{t} \l\langle m_2(s) \times \Delta m(s) , m(s) \r\rangle_{L^2} \, ds - \int_{0}^{t}  \l\langle  \l( 1 + \l| m_1(s) \r|_{\mathbb{R}^3}^2 \r) m(s)  , m(s) \r\rangle_{L^2}  \, ds \\
				\nonumber & - \int_{0}^{t}  \l\langle  \l(  \l| m_1(s) \r|_{\mathbb{R}^3}^2 -  \l| m_2(s) \r|_{\mathbb{R}^3}^2 \r) m_2(s)  , m(s) \r\rangle_{L^2}  \, ds \\
				\nonumber & +  \int_{0}^{t}  \l\langle \int_{\mathbb{U}} \bigl[ L(m_1(s),v) - L(m_2(s),v) \bigr] \, q_s(dv) , m(s) \r\rangle_{L^2}  \, ds \\
				\nonumber & + \frac{1}{2} \int_{0}^{t}  \l\langle \bigl( m(s) \times h \bigr) \times h , m(s) \r\rangle_{L^2}  \, ds \\
				\nonumber & + \frac{1}{2} \int_{0}^{t} \l| m(s) \times h \r|_{L^2}^2 \, ds + \int_{0}^{t}  \l\langle \bigl( m(s) \times h \bigr) , m(s) \r\rangle_{L^2}  \, dW(s) \\
				= & \sum_{i=1}^{9} I_{i}(t).
			\end{align}		
			
			We show estimates only for the sixth term $I_6$.
			By Assumption \ref{assumption Main Assumption}, there exists a constant $C>0$ such that
			\begin{align}
				\nonumber \l| \l\langle \int_{\mathbb{U}} \bigl[ L(m_1(s),v) - L(m_2(s),v) \bigr] \, q_s(dv) , m(s) \r\rangle_{L^2} \r| \leq & \int_{\mathbb{U}} \bigl| L(m_1(s),v) - L(m_2(s),v) \bigr|_{L^2} \, q_s(dv) \l| m(s) \r|_{L^2} \\
				\leq & C \l| m(s) \r|_{L^2}^2 \int_{\mathbb{U}} \kappa(s,v) q_s(dv),\ s\in[0,T].
			\end{align}

			Let $\Phi_C:[0,T] \to \mathbb{R}$ be defined as follows, for some constant $C>0$.
			\begin{align}\label{eqn pathwise uniqueness definition of Phi c d 1}
				\Phi_C(s) = C \l[ 1 + \l|m_2(s)\r|_{L^{\infty}} \l( \l|m_1(s)\r|_{L^{\infty}} + \l|m_2(s)\r|_{L^{\infty}} \r) + \l| \nabla m_2(s) \r|_{L^2}^4 + \int_{\mathbb{U}} \kappa(s,v) q_s(dv) \r].
			\end{align}
			Note that the above function works for dimension $d = 1$. For $d = 2$, the calculations follow by doing a little change in the function $\Phi_c$, which we state below.
			\begin{align}\label{eqn pathwise uniqueness definition of Phi c d 2}
				\Phi_C(s) = C \l[ 1 + \l|m_2(s)\r|_{L^{\infty}} \l( \l|m_1(s)\r|_{L^{\infty}} + \l|m_2(s)\r|_{L^{\infty}} \r) + \l| \nabla m_2(s) \r|_{L^2}^2 \l|m(s)\r|_{H^2}^2  + \int_{\mathbb{U}} \kappa(s,v) q_s(dv)  \r].
			\end{align}
			
			By the assumptions on $m_1,m_2$ (see Theorem \ref{theorem existence of weak martingale solution}), along with Assumption \ref{assumption Main Assumption}, the function $\Phi_C$ is integrable for both $d = 1,2$.
			One can show that there exists a constant $C>0$ such that
			\begin{align}
				\l| m(t) \r|_{L^2}^2  +  \int_{0}^{t} \l| \nabla m(s) \r|_{L^2}^2 \, ds \leq \int_{0}^{t} \Phi_C(s) \l| m(s) \r|_{L^2}^2 \, ds.
			\end{align}
			
			\dela{Applying the Gronwall inequality gives the required inequality, thus concluding the proof.}
			An application of Gronwall's inequality completes the proof.

		\end{proof}

		\bibliographystyle{plain}
		\bibliography{References}
	\end{document}